\documentclass{amsart}
\usepackage[T1]{fontenc}
\usepackage[left=1in, right=1in]{geometry}
\usepackage{graphicx} 
\usepackage{amsmath}
\usepackage{amsthm}
\usepackage{amssymb}
\usepackage{url}
\usepackage{dsfont}
\usepackage{float}
\usepackage{hyperref}
\usepackage[frak=esstix]{mathalpha}
\usepackage{mathtools}
\usepackage[english]{babel}
\theoremstyle{plain}
\newtheorem{theorem}{Theorem}[section]

\newtheorem{lemma}[theorem]{Lemma}

\newtheorem{proposition}[theorem]{Proposition}
\theoremstyle{definition}
\newtheorem{definition}[theorem]{Definition}
\newtheorem{remark}{Remark}[section]

\newtheorem*{theorem*}{Theorem}
\newcommand{\interior}[1]{%
  {\kern0pt#1}^{\mathrm{o}}%
}

\newcommand{\ssubset}{\subset\joinrel\subset}
\numberwithin{equation}{section}

\title[Stable Type I blow-up for the 1D wave equation]{Stable Type I blow-up for the one-dimensional wave equation with time-derivative nonlinearity}
\author{Oliver Gough}
\begin{document}
\begin{abstract}
We study finite-time blow-up for the one-dimensional nonlinear wave equation with a quadratic time-derivative nonlinearity,
\[
u_{tt}-u_{xx}=(u_t)^2,\qquad (x,t)\in\mathbb R\times[0,T).
\]
Building on the work of Ghoul, Liu, and Masmoudi \cite{ghoul2025blow} on the spatial-derivative analogue, we establish the non-existence of smooth, exact self-similar blow-up profiles. Instead we construct an explicit family of \emph{generalised self-similar} solutions, bifurcating from the ODE blow-up, that are smooth within the past light cone and exhibit type-I blow-up at a prescribed point \((x_0,T)\). We further prove asymptotic stability of these profiles under small perturbations in the energy topology.
\end{abstract}

\maketitle
\section{Introduction}
We consider the Cauchy problem for the nonlinear wave equation with quadratic time-derivative nonlinearity in one space dimension,
\begin{equation}
    u_{tt} - u_{xx} = (u_t)^2, 
    \qquad (x,t)\in \mathbb{R}\times [0,T),
    \label{PDE}
\end{equation}
for an unknown function $u:\mathbb{R}\times [0,T)\to\mathbb{R}$.  
This equation admits the scaling and translation invariances
\begin{equation*}
    u(x,t)\mapsto 
    u\!\left(\frac{x-x_0}{\lambda},\,\frac{t-t_0}{\lambda}\right)
    + \kappa,
    \qquad x_0,t_0,\kappa\in\mathbb{R},\ \lambda>0,
\end{equation*}
which will play a central role in the construction of self-similar blow-up solutions. Local well-posedness in $H^{k+1}(\mathbb{R})\times H^{k}(\mathbb{R})$, finite speed of propagation, and persistence of regularity follow from standard energy methods for semilinear wave equations; see, e.g., \cite{sogge1995lectures,evans2022partial}. 
\subsection*{Main results}
As in the power–nonlinearity case, blow-up already occurs for spatially homogeneous data. If $u=u(t)$, then \eqref{PDE} reduces to $u_{tt}=(u_t)^2$. Writing $v:=u_t$ gives $v_t=v^2$, hence $v(t)=\frac{1}{T-t}$ and
\[
u(t)=-\log(T-t)+\kappa,
\]
which exhibits the Type~I rate. Here we loosely define Type I blow-up as blow-up occurring at a rate comparable to that of the ODE blow-up profile. By time-translation and addition of constants, this yields the two-parameter ODE-type family
\[
u_{T,\kappa}(t)=-\log\!\left(1-\frac{t}{T}\right)+\kappa,
\qquad \partial_t u_{T,\kappa}(t)=\frac{1}{T-t},
\]
blowing up at time $T>0$ uniformly in space.

Moreover, by a standard cutoff-and-finite-speed of propagation argument one can prescribe the blow-up point and make the initial data smooth, compactly supported and in particular $H^{k+1}(\mathbb{R})\times H^k(\mathbb{R})$. Fix $T>0$ and $x_0\in\mathbb{R}$, choose $R>T$, and let $\chi\in C_c^\infty(\mathbb{R})$ satisfy $\chi\equiv1$ on $B_R(x_0)$. For the Cauchy problem \eqref{PDE} with initial data
\[
(u,\partial_t u)\big|_{t=0}=\big(\kappa\,\chi(x),\,T^{-1}\chi(x)\big),
\]
finite speed of propagation implies that for all $(x,t)$ in the past light cone
\[
\Gamma(x_0,T):=\{(x,t)\in\mathbb{R}\times[0,T):\,|x-x_0|\leq T-t\}
\]
the solution coincides with the spatially homogeneous ODE profile:
\[
u(x,t)=-\log\!\left(1-\frac{t}{T}\right)+\kappa.
\]
Thus the solution blows up at $(x_0,T)$ while the initial data are compactly supported. As in this example, it suffices to analyse the dynamics inside $\Gamma(x_0,T)$.

Unlike power–type semilinear wave equations, \eqref{PDE} is \emph{not} invariant under Lorentz transformations. In particular, Lorentz boosts
\[
(x,t)\mapsto \big(\gamma(x-vt),\,\gamma(t-vx)\big),\qquad \gamma=\frac{1}{\sqrt{1-v^2}},\quad |v|<1,
\]
do not generate new solutions from ODE blow-up profiles. Nevertheless, following Ghoul–Liu–Masmoudi \cite{ghoul2025blow}, one can still construct solutions to \eqref{PDE} that are effectively Lorentz-boosted ODE blow-up inside $\Gamma(x_0,T)$. 

We first search for exactly self-similar blow-up solutions. Guided by scaling and translation symmetries, consider
\begin{equation*}
u(x,t)=U\!\left(\frac{x-x_0}{T-t}\right)=:U(y),\qquad y:=\frac{x-x_0}{T-t}.
\end{equation*}
A substitution into \ref{PDE} yields the similarity ODE
\begin{equation}
(y^2-1)\,U''(y)+2y\,U'(y)=y^2\big(U'(y)\big)^2.
\label{eq:sim_ode}
\end{equation}
The first result is that apart from the symmetry-induced constant profiles $U\equiv\kappa$, there are no nontrivial, smooth, exact self-similar solutions of \eqref{eq:sim_ode} on $\{|y|\leq1\}$. 

We now introduce the temporal similarity variable
\[
\tau:=-\log\!\left(1-\frac{t}{T}\right)= -\log(T-t)+\log T.
\]
where we observe that this is also the ODE blow up profile. 
Together with the spatial similarity variable \(y:=\frac{x-x_0}{T-t}\), then \((x,t)\mapsto (y,\tau)\) is a smooth diffeomorphism from the interior of \(\Gamma(x_0,T)\) onto \((-1,1)\times[0,\infty)\). Exact self-similarity corresponds to stationarity in these variables, which is ruled out below. Nevertheless, setting
\begin{equation}
u(x,t)=p\,\tau+\widetilde U(y),\qquad p\in\mathbb{R},
\label{eq:gss_ansatz}
\end{equation}
does produce explicit smooth solutions in the cone.

\begin{theorem}[Non-existence of exact smooth self-similar blow-up solutions]\label{thm:existence}
1) For any $T>0$ and $x_0\in\mathbb{R}$, there are no nontrivial smooth exact self-similar blow-up solutions to \eqref{PDE} in the past light cone
\[
\Gamma(x_0,T):=\{(x,t)\in\mathbb{R}\times[0,T):\,|x-x_0|\le T-t\} \equiv \{|y|\le 1\}.
\]
Equivalently, apart from the symmetry-induced constant profiles $\widetilde U \equiv\kappa$, no smooth profile of the form $u(x,t)=\widetilde U (y)$ exists in $\{|y|\le 1\}$.

\smallskip
2) There exists a five-parameter family of smooth generalised self-similar blow-up solutions to \eqref{PDE} with logarithmic growth,
\[
u_{p,q,\kappa,T,x_0}(x,t)
= -p\log\!\Big(1-\frac{t}{T}\Big)\;-\;p\log\!\Big(1+q\sqrt{1-p}\,\frac{x-x_0}{T-t}\Big)\;+\;\kappa,
\]
smooth in the past light cone $\Gamma(x_0,T)$ for all $T>0$, $\kappa,x_0\in\mathbb{R}$, $q\in\{\pm1\}$, and $p\in(0,1]$. In terms of the similarity variables, these solutions read
\[
u_{p,q,\kappa,T,x_0}(x,t)= U_{p,q,\kappa}(\tau,y) = p\,\tau+\widetilde U_{p,q,\kappa}(y),
\qquad \widetilde U_{p,q,\kappa}(y)=-p\log\!\big(1+q\sqrt{1-p}\,y\big)+\kappa.
\]
 
In particular,
\[
u_{1,q,\kappa,T,x_0}(x,t)=-\log\!\Big(1-\frac{t}{T}\Big)+\kappa,
\]
recovers the spatially homogeneous ODE-type blow-up, and the two signs are related by spatial reflection,
\[
u_{p,-1,\kappa,T,x_0}(x,t)=u_{p,1,\kappa,T,-x_0}(-x,t).
\]
\end{theorem}
\begin{remark}[Lorentz boost viewpoint]
Under the boost $t'=\frac{t-\gamma x}{\sqrt{1-\gamma^2}}$, $x'=\frac{x-\gamma t}{\sqrt{1-\gamma^2}}$ with $\gamma\in(-1,1)$, taking the same \emph{spatially homogeneous} ODE blow up profile in the \emph{primed frame} and mapping back yields
\begin{equation*}
    u(x,t) = -(1-\gamma^2)\log \left(T-t-\gamma (x-x_0) \right) + c_2
\end{equation*}
Moreover, since $\gamma \in (-1,1)$ we see that the above is equivalent to 
\begin{equation*}
    u(x,t) = -(1-\gamma^2)\log \left(T-t+q|\gamma| (x-x_0) \right) + c_2
\end{equation*}
for $q\in \{-1,1\}$. A simple relabeling of $p=1-\gamma^2 \in (0,1]$ recovers the solution
\begin{equation*}
    u_{p,q,\kappa,x_0,T}(x,t) = -p\log\left(T-t + q\sqrt{1-p}(x-x_0)\right) + \kappa 
\end{equation*}
\label{remark_lorentz_transformation}
Equivalently, the travelling-wave ansatz $u(x,t)=F(x-ct)$ reduces \eqref{PDE} to the Riccati-ODE mechanism $(c^2-1)F''=c^2(F')^2$, giving
\[
F(\xi)=-(1-c^2)\log(A+B\xi)+\kappa.
\]
Applying translational and scaling symmetries, we choose $A,B$ so that $A+B(x-ct)=T-t-c(x-x_0)$ reproduces the same family. Blow-up occurs along $T-t-c(x-x_0)=0$. See section \ref{sec:lorentz_trans}.
\end{remark}
\begin{remark}[Values of $p$]
For $p<0$, the functions $u_{p,q,\kappa,T,x_0}$ still solve \eqref{PDE}, but the argument of the logarithm vanishes at
\[
y=\pm \frac{1}{\sqrt{\,1-p\,}},
\]
which lies strictly inside the cone \(|y|<1\). Thus a singular point appears inside $\Gamma(x_0,T)$, and by persistence of regularity these profiles cannot arise from smooth initial data; we thus exclude them from the stability analysis.
\end{remark}
We now present the main theorem of asymptotic stability for such generalised self similar blow up solutions. 
\begin{theorem}[Asymptotic stability of generalised self-similar solutions]\label{Stability_main_theorem}
Let $p_0\in(0,1)$, $T_0>0$, $\kappa_0,x_0\in\mathbb{R}$, and $k\ge4$. There exists $\omega_0\in(0,\tfrac12)$ such that for any $\delta\in(0,\omega_0)$ there exists $\varepsilon>0$ with the following property. 

For any real-valued $(f,g)\in H^{k+1}(\mathbb{R})\times H^k(\mathbb{R})$ with
\[
\|(f,g)\|_{H^{k+1}(\mathbb{R})\times H^k(\mathbb{R})}\le \varepsilon,
\]
there exist parameters $p^*\in(0,1)$, $T^*>0$, $\kappa^*\in\mathbb{R}$ and a unique solution $u:\Gamma(x_0,T^*)\to\mathbb{R}$ to \eqref{PDE} with initial data 
\[
u(x,0)=u_{p_0,1,\kappa_0,x_0,T_0}(x,0)+f(x),\qquad
\partial_t u(x,0)=\partial_t u_{p_0,1,\kappa_0,x_0,T_0}(x,0)+g(x)
\quad (x\in B_{T^*}(x_0)).
\]
Then for all $t\in[0,T^*)$,
\begin{align*}
(T^*-t)^{-\frac12+s}\,
\|u(\cdot,t)-u_{p^*,1,\kappa^*,x_0,T^*}(\cdot,t)\|_{\dot H^s(B_{T^*-t}(x_0))}
&\lesssim (T^*-t)^{\omega_0-\delta}, \quad s=0,1,\dots,k+1,\\
(T^*-t)^{-\frac12+s}\,
\|\partial_t u(\cdot,t)-\partial_t u_{p^*,1,\kappa^*,x_0,T^*}(\cdot,t)\|_{\dot H^{s-1}(B_{T^*-t}(x_0))}
&\lesssim (T^*-t)^{\omega_0-\delta}, \quad s=1,\dots,k,
\end{align*}
and the parameters obey
\[
|p_0-p^*|+|\kappa_0-\kappa^*|+\Big|1-\frac{T_0}{T^*}\Big|\lesssim \varepsilon.
\]
\end{theorem}

\begin{remark}[Symmetry $q=\pm1$]
Since $u_{p,1,\kappa,x_0,T}(x,t)=u_{p,-1,\kappa,-x_0,T}(-x,t)$, the conclusion of Theorem~\ref{Stability_main_theorem} also holds for the solutions with $q=-1$.
\end{remark}

\begin{remark}[Similarity normalisation]
The prefactor $(T^*-t)^{-\frac12+s}$ arises from renormalising the $\dot H^s$ norm in similarity variables. Writing $\tau:=-\log(1-t/T^*)$, the right-hand side satisfies
\[
(T^*-t)^{\omega_0-\delta}=(T^*)^{\omega_0-\delta}\,e^{-(\omega_0-\delta)\tau},
\]
so the perturbation decays \emph{exponentially} at rate $\omega_0-\delta>0$ in $\tau$.
\end{remark}
\subsection{Related work}
\medskip

\noindent\textbf{Power-type semilinear wave equations.}
For power nonlinearities $F(u)=|u|^{p-1}u$, the blow-up theory in 1D is by now very well developed, with universality of rates and stability of self-similar profiles established in a series of works by Merle–Zaag and collaborators (see, e.g., \cite{merle2003determination,merle2005determination,merle2007existence,merle2015stability,merle2016dynamics}). In higher dimensions for sub-conformal exponents, the blow-up rate is known and self-similar profiles are stable; see the references above. In the super-conformal regime, spectral/semigroup methods initiated by Donninger and collaborators \cite{donninger2010nonlinear,donninger2012stable,donninger2014stable,donninger2016blowup,chatzikaleas2019stable,donninger2017strichartz,donninger2020blowup,wallauch2023strichartz} yield stability of ODE-type blow-up and related self-similar solutions. Related constructions and stability phenomena also appear in other wave-type equations, including wave maps, Yang–Mills, and the Skyrme model. For energy-critical and supercritical equations, a distinct \emph{Type~II} mechanism—slower-than-ODE blow-up via concentration of stationary states—has been developed in seminal works starting with Krieger–Schlag–Tataru \cite{krieger2009slow} and extended in \cite{hillairet2012smooth,jendrej2017construction,krieger2015instability,krieger2020stability,collot2018type,collot2018singularity,collot2024soliton}.

\medskip
\noindent\textbf{Derivative nonlinear wave equations.}
Explicit blow-up profiles for wave equations with quadratic \emph{derivative} nonlinearities have been largely absent from the literature. The analysis of blow-up for such equations dates back to the seminal contributions of Fritz John \cite{john1979blow,john1985blow} and Klainerman \cite{klainerman1986null}, the latter introducing the celebrated null condition. Among several foundational results, John established blow-up phenomena for semilinear wave equations in $\mathbb{R}^{3+1}$ \cite{john1979blow}. In particular, for the model \ref{PDE} in  $\mathbb{R}^{3+1}$,
he proved that any non-trivial, smooth, compactly supported initial data giving rise to a global solution must in fact produce the trivial solution $u\equiv 0$.
Recent work has primarily focused on lifespan estimates for blow-up solutions. For equations with spatial derivative nonlinearities, see the results of Rammaha \cite{rammaha1995upper,rammaha1997note}, Sasaki and collaborators \cite{sasaki2023lifespan}, and the recent blow-up results of Shao–Takamura–Wang \cite{shao2024blowup}. For the time-derivative model most closely related to our work, Sasaki \cite{sasaki2018regularity} studied the geometry and regularity of the blow-up curve; see also \cite{ishiwata2020blowup,sasaki2021convergence,sasaki2022characteristic}. In the quasilinear setting, Speck \cite{speck2020stable} constructed a stable ODE-type blow-up using energy methods, providing one of the first rigorous examples of stable type-I blow-up for a wave equation with a time-derivative nonlinearity.

Regarding explicit profiles, Ghoul–Liu–Masmoudi \cite{ghoul2025blow} were the first to develop a full construction-and-stability theory for a derivative nonlinear wave equation. For the 1D spatial-derivative model $u_{tt}-u_{xx} = (u_x)^2$, they proved nonexistence of smooth exact self-similar solutions, constructed explicit generalised self-similar blow-up solutions, and established their asymptotic stability. Their proof advances the spectral/semigroup program of Donninger and collaborators \cite{donninger2010nonlinear,donninger2012stable,donninger2014stable,donninger2016blowup,chatzikaleas2019stable,donninger2017strichartz,donninger2020blowup,wallauch2023strichartz} to handle a non-self-adjoint linearised operator with unstable eigenvalues and, crucially, \emph{non-compact} perturbations. Key ingredients include a Lorentz transformation in self-similar variables (yielding a spectral equivalence), an adapted functional framework, and a novel resolvent estimate. We apply their scheme to the present model, obtaining analogous explicit solutions and highlighting both the parallels and differences with their setting.
To the best of our knowledge, this work provides the first explicit blow-up profile for the model \ref{PDE} that is different from the simple ODE blow-up solution. In addition, the travelling-wave viewpoint offered here gives a direct method for constructing such plane-wave blow-up solutions in higher dimensions.

Finally, we note that after the first version of this manuscript was posted on arXiv, Liu and Raees \cite{raees2025stable} obtained similar results for nonlinear wave equations with quadratic time-derivative nonlinearities, including the model considered here, as well as the null-form nonlinearity $(u_t)^2 - (u_x)^2$.
 
\subsection{Outline of the stability proof}\label{sec:outline-stability} We follow the approach of Ghoul–Liu–Masmoudi \cite{ghoul2025blow}, who extend Donninger’s method for non-compact perturbations; for completeness we briefly recall the key steps of their scheme. After constructing the explicit profiles, we pass to similarity coordinates
\[
\tau:=-\log\!\left(1-\frac{t}{T}\right),\qquad y:=\frac{x-x_0}{T-t},
\]
under which \eqref{PDE} becomes
\begin{equation}
U_{\tau\tau}+U_{\tau}+2y\,U_{\tau y}+(y^2-1)U_{yy}+2y\,U_y
=\big(U_{\tau}+y\,U_y\big)^2,
\qquad (y,\tau)\in(-1,1)\times[0,\infty).
\label{similarityPDE}
\end{equation}
With \(U_{p,1,\kappa}(\tau,y)=p\,\tau+\widetilde U_{p,1,\kappa}(y)\) a solution of \eqref{similarityPDE}, we set \(U=U_{p,1,\kappa}+\eta\) and linearise. By introducing the first-order state
\[
\mathbf q:=\begin{pmatrix}q_1\\ q_2\end{pmatrix}
:=\begin{pmatrix}\eta\\ \eta_{\tau}+y\,\eta_y\end{pmatrix},
\]
we have the linearised problem as an abstract Cauchy problem.
\begin{equation}
\partial_\tau\begin{pmatrix}q_1\\[1mm] q_2\end{pmatrix}
=
\begin{pmatrix}
-y\,\partial_y q_1 + q_2\\[2mm]
\partial_{yy}q_1 - y\,\partial_y q_2 + \big(\tfrac{2p}{1+y\sqrt{1-p}} - 1\big)\,q_2
\end{pmatrix}.
\label{first-order-system}
\end{equation}
It is convenient to decompose the linearised operator as
\[
\widetilde{\mathbf L}_p\mathbf q
=\mathbf L\,\mathbf q+\mathbf L_{p,1}\,\mathbf q,\qquad
\mathbf L\begin{pmatrix}q_1\\ q_2\end{pmatrix}
=
\begin{pmatrix}
-y\,\partial_y q_1 + q_2\\
\partial_{yy}q_1 - y\,\partial_y q_2 - q_2
\end{pmatrix},\qquad
\mathbf L_{p,1}\begin{pmatrix}q_1\\ q_2\end{pmatrix}
=\begin{pmatrix}
0\\ \tfrac{2p}{1+y\sqrt{1-p}}\,q_2
\end{pmatrix}
\]
where \(\mathbf L\) is the first order wave operator in similarity variables, while \(\mathbf L_{p,1}\) is the \(p\)-dependent (bounded but \emph{non-compact}) perturbation. This sets the stage for the stability analysis.
\begin{enumerate}
\item \textbf{Mode stability.}
Because the similarity solutions are built from symmetries, the linearised operator carries \emph{unstable} eigenvalues associated to the scaling and translational invariances : $\lambda=1$ and $\lambda=0$. However, these are not genuine instabilities—they correspond to the free choice of parameters. Mode stability consists in showing that for $\mathfrak{Re}\lambda\ge0$, the only eigenvalues are precisely these symmetry modes. This reduces to an ODE eigenvalue problem for \eqref{first-order-system}. As noted in the aforementioned literature, a convenient route is to invoke a Lorentz transform: although \eqref{PDE} itself is not Lorentz invariant, the discrete spectrum is preserved under the change to the primed frame, and, in view of Remark~\ref{remark_lorentz_transformation}, our profiles arise from ODE blow-up in that frame. We refer the interested reader to \cite{merle2007existence,liu2025note} for a discussion of the Lorentz transformation in this setting. We therefore analyse the eigenvalue ODE in the primed variables, where the ODE analysis becomes more tractable.
\item \textbf{Semigroup generation via the free wave operator.}
Our first step is to show that the linearised operator $\widetilde{\mathbf L}_p$ generates a strongly continuous ($C_0$) semigroup. We first prove that the free wave operator $\mathbf L$ generates a $C_0$ semigroup and then invoke the bounded–perturbation theorem (Engel–Nagel \cite{engel2000one}, p.~158) to conclude the result for $\widetilde{\mathbf L}_p$. The proof uses the Lumer–Phillips theorem (\cite{engel2000one}, Theorem 3.15) together with a coercive energy norm equivalent to $H^{k+1}(-1,1)\times H^k(-1,1)$, since on a bounded interval the standard wave energy is only a seminorm. As the argument is essentially the same as in Ghoul–Liu–Masmoudi \cite{ghoul2025blow} and, in higher dimensions, Ostermann \cite{ostermann2024stable}, we state the results without proof.
\item \textbf{A new decomposition to recover compactness}
One of the difficulties for PDEs such as \eqref{PDE} is that the perturbation \(\mathbf{L}_{p,1}\) is not compact, unlike in the power–nonlinearity case. Worse still, the perturbation is not compact relative to \(\mathbf{L}\), which complicates the analysis of the essential spectrum. The key advancement in \cite{ghoul2025blow} was to use a subcoercivity estimate employed by Merle et al.\ \cite{merle2022blow} for the nonlinear Schr\"odinger equation. Since dissipativity is not guaranteed at lower Sobolev regularity, we must work at higher regularity. The subcoercivity estimate allows us to decompose the operator as the sum of a maximally dissipative operator and a finite-rank (compact) projection. This restores compactness properties and yields sufficient control of the spectrum. 

\item \textbf{The linearised evolution}
The above shows that, aside from the symmetry modes, the spectrum lies \emph{strictly} in the left half-plane; this strictness is exactly the spectral gap we need. This is one of the adaptations in this paper. In the Ghoul–Liu–Masmoudi paper, they prove that for all \(\mathfrak{Re}\lambda>-1\) the only eigenvalues are \(0\) and \(1\). We are able to prove this only for \(\mathfrak{Re}\lambda\ge 0\). In order to exhibit the \(\omega_0>0\) such that
\[
\sigma(\mathbf{L}_p)\subset \{\mathfrak{Re}\lambda \le -\omega_0\}\cup\{0,1\}.
\]
we use standard spectral theory results, providing an alternative route to the spectral gap. In our proposition we also have eigenfunctions arising from symmetries at \(0\) and \(1\), and a \emph{generalised} eigenfunction resulting from the parameter \(p\in(0,1)\). Unlike the power–nonlinearity case, the algebraic multiplicity exceeds the geometric multiplicity for the eigenvalue \(0\). To make this precise, we introduce the usual Riesz projections
\begin{equation*}
    \mathbf{P}_{\lambda,p} \coloneqq \frac{1}{2\pi i}\int_{\gamma_\lambda}(z\mathbf{I}-\mathbf{L}_p)^{-1}\,dz
\end{equation*}
for suitable contours. We prove, using a direct ODE analysis, that in a sufficiently high-regularity Sobolev space there are no further generalised eigenfunctions. To then pass from spectral information to semigroup decay on the stable subspace, we estimate the resolvent and invoke the Gearhart–Pr\"uss–Greiner Theorem (\cite{engel2000one}, Theorem~1.11, p.~302). This resolvent estimate is obtained as in \cite{ghoul2025blow}. After this, a full description of the linearised evolution is given.

\item \textbf{Nonlinear stability}
The full nonlinear system for the perturbation \(\mathbf{q}\) is
\begin{equation*}
    \begin{cases}
        \partial_{\tau}\mathbf{q} = \mathbf{L}_p \mathbf{q} + \mathbf{N}(\mathbf{q}),\\[1mm]
        \mathbf{q}(0,y)=\mathbf{q}_0(y),
    \end{cases}
\end{equation*}
where
\begin{equation*}
    \mathbf{N}(\mathbf{q})=\begin{pmatrix}
        0\\[1mm]
        q_2^2
    \end{pmatrix}.
\end{equation*}
By Duhamel’s principle,
\begin{equation*}
    \mathbf{q}(\tau,\cdot)
    = \mathbf{S}_p(\tau)\,\mathbf{q}_0(\cdot)
      + \int_{0}^{\tau}\mathbf{S}_p(\tau-\tau')\,\mathbf{N}\big(\mathbf{q}(\tau',\cdot)\big)\,d\tau' .
\end{equation*}
Because the symmetry and the generalised eigenvector induce linear growth, this Duhamel formulation does not close globally on \([0,\infty)\) without an adjustment. To remedy this, we follow the Lyapunov–Perron technique: we project the initial data onto the stable subspace and introduce a correction along the unstable directions to cancel their growth. Concretely, we consider the modified problem
\begin{equation*}
    \mathbf{q}(\tau,\cdot)
    = \mathbf{S}_p(\tau)\!\left(\mathbf{U}_{p,T,\kappa}(\mathbf{f})
      - \mathbf{C}_p\big(\mathbf{U}_{p,T,\kappa}(\mathbf{f}),\,\mathbf{q}(\tau,\cdot)\big)\right)
      + \int_{0}^{\tau}\mathbf{S}_p(\tau-\tau')\,\mathbf{N}\big(\mathbf{q}(\tau',\cdot)\big)\,d\tau',
\end{equation*}
where \(\mathbf{U}_{p,T,\kappa}\) 
  maps the initial data into similarity coordinates and incorporates a nearby choice of parameters  thereby expressing the data as a perturbation of a modulated profile. We then show that there exist parameters \(p^*,T^*,\kappa^*\) for which the correction term vanishes, yielding a genuine solution of the fixed-point problem. Proving that the correction vanishes is equivalent to showing that a certain linear functional—obtained by dual pairing with the correction term—is identically zero. Appealing to Brouwer’s fixed-point theorem exhibits the required parameters. Finally, by the restriction properties of the semigroup, the resulting mild solution upgrades to a classical solution.
\end{enumerate}
\subsection{Notation}
In this paper we follow the same notation as in \cite{ghoul2025blow}.
For $x_0\in\mathbb R$ and $r>0$, we denote the open interval centered at $x_0$ by
\[
B_r(x_0):=(x_0-r,\,x_0+r).
\]
Its closure is $\overline{B_r(x_0)}=[x_0-r,\,x_0+r]$. For $x_0\in\mathbb R$ and $T>0$, the past light cone is
\[
\Gamma(x_0,T):=\{(t,x)\in[0,T)\times\mathbb R:\ |x-x_0|\le T-t\}.
\]

For $a,b\in\mathbb R$, we write $a\lesssim b$ if there exists a constant $C>0$ (independent of the variables under consideration) such that $a\le C\,b$. If $C$ may depend on a parameter $p$, we write $a\lesssim_p b$.

For a function $f:\mathbb R\to\mathbb R$, $y\mapsto f(y)$, we may interchangeably denote the first derivative by
$f'=\partial_y f=f_y$ and the $k$-th derivative by $\partial_y^k f=\partial^k f$, where $\partial^k$ without subscript always means $\partial_y^k$. If $\Omega\subset\mathbb R$ is a domain, then $C^\infty(\overline{\Omega})$ denotes smooth functions on $\Omega$ with derivatives continuous up to $\partial\Omega$. If $\Omega$ is bounded and $k\in\mathbb N\cup\{0\}$, we set
\[
\|f\|_{H^k(\Omega)}^2:=\sum_{i=0}^k \|\partial_y^i f\|_{L^2(\Omega)}^2,
\qquad
\|f\|_{\dot H^k(\Omega)}:=\|\partial_y^k f\|_{L^2(\Omega)},
\quad f\in C^\infty(\overline{\Omega}).
\]
The Sobolev space $H^k(\Omega)$ is the completion of $C^\infty(\overline{\Omega})$ with respect to $\|\cdot\|_{H^k(\Omega)}$. For $k\ge0$, we use the product (energy) space
\[
\mathcal{H}^k(-1,1):=H^{k+1}(-1,1)\times H^k(-1,1).
\]
We may also omit the domain (e.g., $(-1,1)$) when it is clear from context.

We use boldface for tuples of functions; e.g.
\[
\mathbf f\equiv(f_1,f_2)^{\!\top},\qquad
\mathbf q(t,\cdot)\equiv\big(q_1(t,\cdot),\,q_2(t,\cdot)\big)^{\!\top}.
\]
Linear operators acting on such tuples are also written in boldface.

If $L$ is a closed linear operator on a Banach space $X$, we denote its domain by $D(L)$, its spectrum by $\sigma(L)$, and its point, discrete, and essential spectrum by $\sigma_{p}(L)$, $\sigma_{\text{disc}}(L)$ and $\sigma_{\text{ess}}(L)$, respectively. The resolvent set is $\rho(L):=\mathbb C\setminus\sigma(L)$ and the resolvent operator is
\[
R_L(z):=(z-L)^{-1},\qquad z\in\rho(L).
\]
The space of bounded operators on $X$ is denoted by $\mathcal L(X)$. For background on spectral theory we refer to Kato \cite{kato2013perturbation}, and for $C_0$-semigroups to Engel and Nagel \cite{engel2000one} and the references therein.
\section{Construction of smooth blow-up solutions}
\subsection{Non-existence of smooth self-similar blow-up}
Since the nonlinear wave equation \eqref{PDE} is not Lorentz invariant, we cannot use Poincar\'e symmetries or Lorentz boosts—as in the power-nonlinearity setting—to generate spatial dependence from an ODE blow-up profile. Nevertheless, as mentioned we construct solutions to \eqref{PDE} that effectively have the form of Lorentz-transformed ODE blow-up inside the cone.

Guided by scaling and translations, we first consider the exactly self-similar ansatz
\[
u(x,t)=U\!\left(\frac{x-x_0}{T-t}\right)=U(y),\qquad y:=\frac{x-x_0}{T-t}.
\]
Substituting into \eqref{PDE} yields the similarity ODE
\begin{equation}
(y^2-1)U_{yy}+2y\,U_y=y^2\,(U_y)^2.
\label{similarityODE_v2}
\end{equation}
Besides the symmetry-induced constants \(U\equiv\kappa\), we can solve \eqref{similarityODE_v2} explicitly for \(V:=U_y\) by rewriting it as
\[
-\left(\frac{1}{V}\right)'+\frac{2y}{y^2-1}\,\frac{1}{V}=\frac{y^2}{y^2-1}.
\]
With \(W:=\frac{1}{V}\) and the integrating factor \((y^2-1)^{-1}\), we obtain
\[
\frac{d}{dy}\!\left(W(y^2-1)^{-1}\right)=-\frac{y^2}{(y^2-1)^2}
=\frac{1}{4(y+1)}-\frac{1}{4(y+1)^2}-\frac{1}{4(y-1)}-\frac{1}{4(y-1)^2}.
\]
Integrating gives
\[
W=\frac{1}{4}\Big(2y+c(y^2-1)+(y^2-1)\log\Big|\frac{y+1}{y-1}\Big|\Big)=:p_c(y),
\]
so
\[
U_y(y)=
\begin{cases}
{4}\left(2y+c(y^2-1)+(y^2-1)\log\!\left|\frac{y+1}{y-1}\right|\right)^{-1},&\text{or}\\[2mm]
0.
\end{cases}
\]

\medskip
\noindent\emph{Proof of Theorem~\ref{thm:existence}(1).}
For any \(c\in\mathbb R\), the denominator \(p_c\) is continuous on \((-1,1)\) and satisfies
\(p_c(-1)=-\tfrac12\) and \(p_c(1)=\tfrac12\).
By the intermediate value theorem, \(p_c\) has a zero in \((-1,1)\), whence \(U_y\) is singular there. Thus there are no nontrivial smooth exact self-similar blow-up profiles in the past light cone \(\Gamma(x_0,T)\equiv\{|y|\le 1\}\).
\qed

\subsection{Construction of generalised self-similar solutions}
Since stationary similarity profiles are ruled out, we perturb the ansatz by allowing a linear growth in the temporal similarity variable
\[
\tau:=-\log\!\left(1-\frac{t}{T}\right)=\log T-\log(T-t).
\]
(Recall that \(u(x,t)=\tau\) is the spatially homogeneous ODE blow-up.) We look for
\[
U(\tau,y)=p\,\tau+\widetilde U(y),\qquad p\in\mathbb R.
\]
Substituting into \eqref{similarityPDE} yields a Riccati equation for \(V:=\widetilde U_y\),
\begin{equation}
(1-y^2)\widetilde U_{yy}+2y(p-1)\widetilde U_y+p(p-1)=-y^2(\widetilde U_y)^2,
\label{perturbed_ode}
\end{equation}
i.e.
\[
V'=Q_0(y)+Q_1(y)V+Q_2(y)V^2,\qquad
Q_0=\frac{p(1-p)}{1-y^2},\quad
Q_1=\frac{2(1-p)y}{1-y^2},\quad
Q_2=-\frac{y^2}{1-y^2}.
\]
A simple ansatz \(V=\frac{a}{b+cy}\) produces the two elementary particular solutions
\[
V_{\pm}(y)=\frac{p\sqrt{1-p}}{\pm 1-y\sqrt{1-p}}.
\]
These are real-valued for \(p\in(0,1]\), and yield the ODE profile when \(p=1\).

\paragraph{The case \(p=1\):}
Setting \(p=1\) in \eqref{perturbed_ode} gives
\[
(1-y^2)\widetilde U_{yy}=-y^2(\widetilde U_y)^2.
\]
Besides the constant solutions \(\widetilde U\equiv\kappa\), any nontrivial solution is singular in \((-1,1)\). Indeed, with \(W:=\frac{1}{\widetilde U_y}\) we have
\[
W'=\frac{y^2}{1-y^2}=-1+\frac{1}{2(1+y)}+\frac{1}{2(1-y)},
\]
hence
\[
\frac{1}{\widetilde U_y}=-y+\frac{1}{2}\log\!\left|\frac{1+y}{1-y}\right|+c=: \rho_c(y),
\]
and \(\rho_c(y)\to\pm\infty\) as \(y\to\pm1\), so \(\rho_c\) has a zero in \((-1,1)\). Thus \(\widetilde U_y\) blows up inside the cone, and only the constant profile can persist from smooth data, recovering the ODE blow-up.

\paragraph{The case \(p<0\).}
Although \(V_\pm\) are real-valued, \(\sqrt{1-p}>1\) and hence \(V_\pm\) blow up at
\(y=\pm 1/\sqrt{1-p}\in(-1,1)\). By persistence of regularity, such profiles cannot evolve from smooth Cauchy data, and we discard \(p<0\).

\paragraph{\textit{General solution for} \(p\in(0,1)\)}:
Given a particular solution, the Riccati equation linearises via
\[
V(y)=V_+(y)+\frac{1}{z(y)},\qquad
z'+\big(Q_1+2Q_2V_+\big)z=-Q_2.
\]
A computation, using partial fractions, shows
\[
z'+\left(\frac{2(1-p)y}{1-y^2}-\frac{2p\sqrt{1-p}\,y^2}{(1-y\sqrt{1-p})(1-y^2)}\right)z=\frac{y^2}{1-y^2},
\]
with integrating factor
\[
\mathcal I(y)=\frac{1}{(1-y\sqrt{1-p})^2}\left(\frac{1-y}{1+y}\right)^{\sqrt{1-p}}.
\]
Integrating yields
\[
z(y)=-\frac{1-y\sqrt{1-p}}{2p\sqrt{1-p}}
\left\{c\left(\frac{1+y}{1-y}\right)^{\sqrt{1-p}}(1-y\sqrt{1-p})+(1-y\sqrt{1-p})\right\},
\]
and hence
\[
\widetilde U_y(y)=V_+(y)+\frac{1}{z(y)}
=\frac{p\sqrt{1-p}}{1-\displaystyle\frac{2}{1+c\,e^{2\sqrt{1-p}\tanh^{-1}(y)}}-y\sqrt{1-p}}.
\]
Thus \(c=0\) reproduces \(V_-\), while \(c=\pm\infty\) gives \(V_+\).

To obtain smoothness on \([-1,1]\), we analyse the denominator
\[
h_c(y):=1-\frac{2}{1+c\left(\frac{1+y}{1-y}\right)^{\sqrt{1-p}}}-y\sqrt{1-p}.
\]
For \(0<c<\infty\) and any \(p\in(0,1)\), \(h_c\) is continuous on \([-1,1]\) and
\[
\lim_{y\to1}h_c(y)=1-\sqrt{1-p},\qquad
\lim_{y\to-1}h_c(y)=-1+\sqrt{1-p},
\]
which have opposite signs; by the intermediate value theorem, \(h_c\) vanishes in \((-1,1)\), so \(\widetilde U_y\) blows up. For \(c<0\), \(\widetilde U_y\) remains continuous but
\[
\widetilde U_{yy}(y)=-\frac{p\sqrt{1-p}\left(\displaystyle\frac{4c\sqrt{1-p}\,e^{2\sqrt{1-p}\tanh^{-1}(y)}}{(1-y^2)\big(c\,e^{2\sqrt{1-p}\tanh^{-1}(y)}+1\big)^2}-\sqrt{1-p}\right)}{\left(1-\dfrac{2}{c\,e^{2\sqrt{1-p}\tanh^{-1}(y)}+1}-y\sqrt{1-p}\right)^2}
\]
is unbounded at \(y=\pm1\).
Consequently, smoothness on the closed cone forces \(c\in\{0,\pm\infty\}\).

Collecting cases, for \(p\in(0,1]\), \(q\in\{\pm1\}\), and \(\kappa\in\mathbb R\), the smooth profiles are
\[
\widetilde U_{p,q,\kappa}(y)=-p\log\big(1+q\sqrt{1-p}\,y\big)+\kappa,
\]
and hence
\[
U_{p,q,\kappa}(\tau,y)=p\,\tau-p\log\big(1+q\sqrt{1-p}\,y\big)+\kappa.
\]
In \((x,t)\)-variables this gives the five-parameter family
\begin{align}
u_{p,q,\kappa,x_0,T}(x,t)
&=-p\log\!\left(1-\frac{t}{T}\right)-p\log\!\left(1+q\sqrt{1-p}\,\frac{x-x_0}{T-t}\right)+\kappa\\
&=-p\log\!\big(T-t+q\sqrt{1-p}\,(x-x_0)\big)+p\log T+\kappa,
\end{align}
with the special case \(p=1\) recovering the ODE blow-up, and
\begin{equation}
u_{p,-1,\kappa,x_0,T}(x,t)=u_{p,1,\kappa,-x_0,T}(-x,t).
\label{Relation}
\end{equation}

\subsection{Solutions obtained from the Lorentz transformation}\label{sec:lorentz_trans}
Although neither \((u_t)^2\) nor \((u_x)^2\) is Lorentz invariant (in contrast with \((u_t)^2-(u_x)^2\)), the Lorentz transform offers a complementary viewpoint on the preceding family.
Let
\[
t'=\frac{t-\gamma x}{\sqrt{1-\gamma^2}},\qquad
x'=\frac{x-\gamma t}{\sqrt{1-\gamma^2}},\qquad \gamma\in(-1,1),
\]
with inverse
\[
t=\frac{t'+\gamma x'}{\sqrt{1-\gamma^2}},\qquad
x=\frac{x'+\gamma t'}{\sqrt{1-\gamma^2}}.
\]
If \(u\) solves \eqref{PDE}, then \(v(x',t'):=u(x(x',t'),t(x',t'))\) satisfies
\[
v_{t't'}-v_{x'x'}=\frac{1}{1-\gamma^2}\Big((v_{t'})^2-2\gamma v_{x'}v_{t'}+\gamma^2(v_{x'})^2\Big).
\]
Considering the ODE mechanism \(v_{t't'}=\frac{1}{1-\gamma^2}(v_{t'})^2\) yields
\[
v(x',t')=-(1-\gamma^2)\log(c_1+t')+c_2.
\]
Transforming back, applying time reversal, and using translations gives
\[
u(x,t)=-(1-\gamma^2)\log\big(T-t-\gamma(x-x_0)\big)+c_2
=-(1-\gamma^2)\log\big(T-t+q|\gamma|(x-x_0)\big)+c_2,
\]
with \(q\in\{\pm1\}\). Relabelling \(p=1-\gamma^2\in(0,1]\) reproduces
\[
u_{p,q,\kappa,x_0,T}(x,t)=-p\log\!\big(T-t+q\sqrt{1-p}(x-x_0)\big)+\kappa,
\]
up to the symmetry constant \(p\log T\).

\begin{remark}
By scaling symmetry, the Lorentz transform corresponds to travelling-wave profiles: the ansatz \(u(x,t)=F(x-ct)\) reduces \eqref{PDE} to the Riccati mechanism \( (c^2-1)F''=c^2(F')^2\), whose solutions are
\(F(\xi)=-(1-c^2)\log(A+B\xi)+\kappa\).
After translations and scaling, choosing \(A+B(x-ct)=T-t-c(x-x_0)\) recovers the same family, with blow-up along \(T-t-c(x-x_0)=0\).
\end{remark}
\section{Mode stability}
We now turn to the stability of the above blow-up solution. First, we perform the most elementary stability analysis of $U_{p,1,\kappa}(\tau,y)$ in similarity coordinates, focusing solely on possible unstable point spectrum. It suffices to consider $q=1$ due to the relation~\eqref{Relation}.
\\

\noindent We substitute $U(\tau,y) = U_{p,1,\kappa}(\tau,y) + \eta(\tau,y)$ into \eqref{similarityPDE} and linearise, obtaining 
\begin{equation}
    \partial_{\tau \tau}\eta + \partial_{\tau} \eta + 2y\partial_{\tau y}\eta + 2y\partial_y \eta + (y^2-1)\partial_{yy}\eta = 2(\partial_{\tau} \eta +y\partial_y \eta)\left(\frac{p}{1+y\sqrt{1-p}}\right).
\label{similaritylinearisedPDE}
\end{equation}
In the same way one puts the free wave equation $\Box u = 0 $ into a first order (in time) evolution equation in the arguments $(u,\partial_t u)$, we let $\mathbf{q}=(q_1,q_2)^T \coloneqq (\eta,\partial_{\tau}\eta +y\partial_y \eta)^T$, and rewrite \eqref{similaritylinearisedPDE} as the first order system
\begin{equation}
\partial_\tau \begin{pmatrix}
q_1 \\
q_2
\end{pmatrix}
=
\begin{pmatrix}
    -y\partial_y q_1 +q_2 \\
    \partial_{yy}q_1 -y\partial_y q_2 +\left(\frac{2p}{1+y\sqrt{1-p}}-1 \right)q_2
\end{pmatrix}
\label{first-order system}.
\end{equation}
We denote the right hand side as the operator  
\begin{equation*}
    \widetilde{\mathbf{L}}_{p}\begin{pmatrix} 
    q_1 \\
    q_2
\end{pmatrix} \coloneqq \begin{pmatrix}
    -y\partial_y q_1 +q_2 \\
    \partial_{yy}q_1 -y\partial_y q_2 + \left(\frac{2p}{1+y\sqrt{1-p}} -1 \right)q_2 
\end{pmatrix} = \mathbf{L}\mathbf{q} + \mathbf{L}_{p,1}\mathbf{q}
\end{equation*}
where 
\begin{equation*}
    \mathbf{L}
    \begin{pmatrix}
    q_1 \\
    q_2
    \end{pmatrix} =
    \begin{pmatrix}
        -y\partial_y q_1 + q_2 \\
        \partial_{yy}q_1 - y\partial_y q_2 - q_2
    \end{pmatrix}
\end{equation*} is the free wave operator and $\mathbf{L}_{p,1}$ 
\begin{equation*}
    \mathbf{L}_{p,1}\begin{pmatrix}
        q_1 \\
        q_2
    \end{pmatrix} \coloneqq \begin{pmatrix}
        0 \\
        \frac{2p}{1+y\sqrt{1-p}}q_2
    \end{pmatrix}
\end{equation*}
is the perturbation which depends only on $p$. \\

\noindent Unfortunately the linearised operator $\widetilde{\mathbf{L}}_p$ not self adjoint on $H^{k+1}(-1,1)\times H^k(-1,1)$ which means in general that $\sigma(\mathbf{L}_p)\subset \mathbb{C}$. As mentioned, the perturbation $\mathbf{L}_{p,1}$ is a bounded operator on $H^{k+1}\times H^k$, it is not relatively compact with respect to $\mathbf{L}$ which also complicates finding information on the essential spectrum of $\widetilde{\mathbf{L}}_p$.\\

\noindent By considering the separated variable solutions $\eta = e^{\lambda \tau}\varphi(y)$ to \eqref{similaritylinearisedPDE} we obtain the eigenequation ODE
\begin{equation}
    \left(\lambda^2 + \lambda - \frac{2p\lambda}{1+y\sqrt{1-p}}\right)\varphi(y) + \left((2\lambda+2)y-\frac{2yp}{1+y\sqrt{1-p}} \right)\varphi'(y) + (y^2-1)\varphi''(y) = 0.
\label{eigenequation}
\end{equation}
Equation \eqref{eigenequation} is called the eigenequation because there is a 1-to-1 correspondence between solutions to \eqref{eigenequation} and solutions $(\mathbf{q},\lambda)$ to $\widetilde{\mathbf{L}}_{p}\mathbf{q} = \lambda \mathbf{q}$, as given by the following proposition.
\begin{proposition}
If \(\mathbf q=(q_1,q_2)\in C^\infty[-1,1]\times C^\infty[-1,1]\) satisfies \(\widetilde{\mathbf L}_p\,\mathbf q=\lambda\,\mathbf q\) for some \(\lambda\in\mathbb C\), then \(q_1\) solves \eqref{eigenequation}. Conversely, if \((\varphi,\lambda)\) solves \eqref{eigenequation} with \(\varphi\in C^\infty[-1,1]\), then
\(\mathbf\Psi:=(\varphi,\ \lambda\varphi+y\,\varphi')^{\!\top}\) satisfies \(\widetilde{\mathbf L}_p\,\mathbf\Psi=\lambda\,\mathbf\Psi\).
\end{proposition}
\begin{proof}
Direct computation: eliminating \(q_2\) from \(\widetilde{\mathbf L}_p\mathbf q=\lambda\mathbf q\) yields \eqref{eigenequation}; conversely, \(\eta(\tau,y)=e^{\lambda\tau}\varphi(y)\) solves \eqref{similaritylinearisedPDE}, hence \((q_1,q_2)=(\eta,\eta_\tau+y\eta_y)\) is an eigenvector, and factoring out \(e^{\lambda\tau}\) gives the stated \(\mathbf\Psi\).
\end{proof}

\noindent As a result, we make the following point–spectrum conventions.
\begin{definition}
We call \(\lambda\in\mathbb C\) an eigenvalue of \(\widetilde{\mathbf L}_p\) if \eqref{eigenequation} admits a nontrivial \(\varphi\in C^\infty[-1,1]\).
\end{definition}
\begin{remark}[Sobolev vs.\ smooth eigenfunctions]\label{rem:pspec-regularity}
It is equivalent to define eigenvalues of $\widetilde{\mathbf L}_p$ using eigenfunctions
in $H^{k+1}(-1,1)$ for $k\ge 4$. Indeed, the eigenequation \eqref{eigenequation} is a
second-order elliptic ODE with smooth coefficients on $(-1,1)$, so any weak
$H^{k+1}$ solution is $C^\infty$ in the interior by standard elliptic regularity
(e.g., Evans \cite{evans2022partial}, Ch.~6). Moreover, we show in
Proposition~\ref{one_dimentsional space appendix} (Appendix~\ref{sec:appendix-regularity})
that such solutions are smooth—indeed real-analytic—up to the endpoints $y=\pm1$.
Hence the $H^{k+1}$ and $C^\infty$ notions of point spectrum coincide.
\end{remark}
\begin{definition}
An eigenvalue \(\lambda\in\mathbb C\) is \emph{unstable} if \(\mathfrak{Re}\lambda\ge 0\); otherwise it is \emph{stable}.
\end{definition}
\subsection{Symmetry–generated eigenmodes}

If \(U(\tau,y)\) solves \eqref{similarityPDE}, the following symmetry actions map solutions to solutions:
\begin{itemize}
  \item Spatial translation: \(\displaystyle U_a(\tau,y)=U\!\big(\tau,\,y+a\,e^\tau\big)\).
  \item Shift of the blow-up time: \(\displaystyle U_b(\tau,y)=U\!\left(\tau-\log(1-b e^\tau),\,\frac{y}{1-b e^\tau}\right)\).
  \item \(\tau\)-translation: \(\displaystyle U_c(\tau,y)=U(\tau+c,\,y)\).
  \item Additive constant: \(\displaystyle U_d(\tau,y)=U(\tau,y)+d\).
\end{itemize}

\begin{remark}
The physical scaling in \((x,t)\) is a composition of a time shift in \(t\) and a \(\tau\)-translation, so it yields no additional independent direction in similarity variables.
\end{remark}

Differentiating at parameter \(0\) gives the symmetry generators
\[
\partial_a U_a\big|_{a=0}=e^\tau U_y,\qquad
\partial_b U_b\big|_{b=0}=e^\tau(U_\tau+yU_y),\qquad
\partial_c U_c\big|_{c=0}=U_\tau,\qquad
\partial_d U_d\big|_{d=0}=1.
\]

For the profile
\[
U_{p,1,\kappa}(\tau,y)=p\tau-p\log\!\bigl(1+y\sqrt{1-p}\,\bigr)+\kappa,\qquad p\in[0,1],
\]
we have
\[
\partial_a U_a\big|_{a=0}= -\,e^\tau\frac{p\sqrt{1-p}}{1+y\sqrt{1-p}},\quad
\partial_b U_b\big|_{b=0}= e^\tau\frac{p}{1+y\sqrt{1-p}},\quad
\partial_c U_c\big|_{c=0}= p,\quad
\partial_d U_d\big|_{d=0}=1.
\]

Thus we have the following eigenvalues and eigenfunctions for \(0<p<1\),
\begin{itemize}
\item \(\partial_a U_a|_{a=0}\) and \(\partial_b U_b|_{b=0}\) are proportional, hence span the \emph{same} one–dimensional unstable eigenspace at eigenvalue \(\lambda=1\).
\item \(\partial_c U_c|_{c=0}\) and \(\partial_d U_d|_{d=0}\) are constants in \(y\) and span the \emph{same} one–dimensional neutral eigenspace at \(\lambda=0\).
\end{itemize}
Equivalently, writing modes as \(e^{\lambda\tau}\varphi(y)\), we may take (up to normalisation)
\[
\varphi_1(y)=\frac{1}{1+y\sqrt{1-p}}\quad(\lambda=1),\qquad
\varphi_0(y)\equiv 1\quad(\lambda=0).
\]

\paragraph{A generalised neutral mode (\(\lambda=0\)).}
Since \(U_{p,1,\kappa}\) also depends continuously on $p$, differentiating in \(p\) produces a generalised mode at \(\lambda=0\). In the first–order formulation this function reads
\[
(q_1,q_2)^\top
=\Bigl(\partial_p U_{p,1,\kappa},\; \partial_p\bigl(U_\tau+yU_y\bigr)\Bigr)^\top
\]
which satisfies $\widetilde{\mathbf{L}}_p\,(q_1,q_2)^\top = (\varphi_0,0)^\top$.

We will say the profile is \emph{mode stable} if the point spectrum in \(\mathfrak{Re}\,\lambda\ge0\) consists \emph{only} of the symmetry generated modes. 

\begin{definition}
    (Mode Stability) The solution $U_{p,1,\kappa}$ is mode stable if the existence of a non-trivial $\varphi_\lambda \in C^\infty[-1,1]$ satisfying \eqref{eigenequation} implies that either $\mathfrak{Re}(\lambda)<0$ or $\lambda\in \{0,1\}$.
\end{definition}
We now begin the study of mode stability of the solutions $U_{p,1,\kappa}$. Rewriting \eqref{eigenequation} with partial fractions 
\begin{equation*}
    \varphi'' + \left(\frac{\lambda - \sqrt{1-p}}{y+1} + \frac{\lambda + \sqrt{1-p}}{y-1}+\frac{2\sqrt{1-p}}{1+y\sqrt{1-p}}\right)\varphi' + \left(\frac{\frac{\lambda^2}{2}-\frac{\lambda}{2}+\lambda\sqrt{1-p}}{y-1}+\frac{-\frac{\lambda^2}{2}+\frac{\lambda}{2}+\lambda\sqrt{1-p}}{y+1} +\frac{-2\lambda(1-p)}{1+y\sqrt{1-p}}\right)\varphi = 0
\end{equation*}
we see the presence of four regular singular points, $\pm 1, \frac{-1}{\sqrt{1-p}}$ and $\infty$. ODEs of this class are called \textit{Heun differential equations}. Unfortunately, the existence of smooth solutions to these ODE, known as the `connection problem' is still open. Therefore, following the method of the aforementioned reference, we transform this Heun ODE into the more tractable hypergeometric equation which possesses only three regular singular points.
\noindent \textit{Note:} When $p=1$ (corresponding to the ODE blow up solution), the coefficients $C, F=0$, reducing it already to the hypergeometric equation
\begin{equation}
    \varphi'' + \left(\frac{\lambda}{y-1}+\frac{\lambda}{y+1}\right)\varphi' + \frac{\lambda^2-\lambda}{y^2-1}\varphi = 0
\label{p=1_hypergeometric}
\end{equation}
with the three singular points $\pm 1$ and $\infty$, an expected phenomenon for eigen-equations associated with ODE blow up solutions. We will treat this case separately since the eigenvalues are explicit. \\

\noindent Before this, we recall a classical ODE result. 

\subsection{Frobenius Theory and Hypergeometric functions}
\begin{lemma}(Frobenius-Fuchs \cite{teschl2012ordinary} \cite{donninger2024spectral})
Consider the ODE over $\mathbb{C}$
\begin{equation}
    f''(z) + p(z)f'(z) + q(z)f(z) = 0. 
    \label{frobenius_ode}
\end{equation}
Denoting the complex ball of radius $R>0$ be $\mathcal{D}_R = \{z\in \mathbb{C}: |z|<R\}$ and where $p,q:\mathcal{D}_{R} \setminus \{0\} \to \mathbb{C}$ are holomorphic. Suppose the limits 
\begin{align*}
    p_0 &= \lim_{z\to 0}[zp(z)] \\
    q_0 &= \lim_{z\to 0}[z^2q(z)]
\end{align*}
exists, i.e \,$zp(z),z^2q(z)$ are analytic at $z=0$, i.e $p,q$ are meromorphic $\mathbb{C}$ with poles at most of order 1 and 2. So we can write 
\begin{align*}
    zp(z) &= \sum_{n=0}^{\infty}p_n z^n \\
    z^2 q(z) &= \sum_{n=0}^{\infty}q_n z^n
\end{align*}
which converge on $|z|<R$. Define also the `indicial polynomial' $P(s)$ as
\begin{equation*}
    P(s) \coloneqq s(s-1)+p_0 s + q_0
\end{equation*}
and denote $s_{j}\in \mathbb{C}$ $j\in \{1,2\}$ the roots $P(s_{j}) = 0$ where $\mathfrak{Re}(s_1) \geq \mathfrak{Re}(s_2)$
\begin{itemize}
    \item If $s_1 - s_2 \notin \mathbb{N}_0 \coloneqq \mathbb{N}\cup \{0\}$, the fundamental system of solutions to \eqref{frobenius_ode} is given by 
    \begin{equation*}
        f_{j}(z) = z^{s_{j}}h_j(z) 
    \end{equation*}
    where the functions $h_j(z) = \sum_{n=0}^{\infty}a_{n}^j z^n$ are some analytic functions at 0 from $\mathcal{D}_R \to \mathbb{C}$ and satisfy $h_{j}(0) = a_0^j = 1$. 
    \item  If $s_1 - s_2 = m \in \mathbb{N}_0$ the fundamental system of solutions to \eqref{frobenius_ode} is 
    \begin{align*}
        f_1(z) &= z^{s_1}h_1(z) \\
        f_2(z) &= z^{s_2}h_3(z) + c\log(z)f_1(z)
    \end{align*}
    where $h_j$ are analytic near $z=0$, satisfying $h_j(0)=1$. The function $h_3 = \sum_{n=0}^{\infty}b_n z^n$ where the coefficients $b_n$ are determined by equating coefficients. Again, we also have $h_3(0) = b_0 =1$. The constant $c\in \mathbb{C}$ may be zero unless $m=0$.
\end{itemize}
Moreover we can obtain the coefficients $a_n^j$ by the recurrence relation 
\begin{equation}
    a_n^j\left[(s_j+n)(s_j+n-1)+(s_j+n)p_0 +q_0\right] + \sum_{k=0}^{n-1}a_k^j\left[(s_j+k)p_{n-k} +q_{n-k} \right] = 0
    \label{recurrence_rel}
\end{equation}
Finally, the radii of convergence of the series $h_j$ in either cases of $s_1-s_2$ is non-integer or integer, including zero, are not less than the distance to the nearest of the next nearest singularity of the differential equation from $z=0$ \cite{olver2010nist}.
\end{lemma}
\noindent \textbf{Hypergeometric equations}\\
\noindent A particular class of \eqref{frobenius_ode} is the \textit{hypergeometric} differential equation which takes the form 
\begin{equation*}
    z(1-z)\frac{d^2 w}{d z^2} + [\gamma-(\alpha+\beta+1)z]\frac{dw}{dz} -abw =0
\end{equation*}
for parameters $a,b,c$ and where $p_0 = c$ and $q_0 = 0$ so that 
\begin{align*}
    zp(z) &= \gamma + (\gamma-\alpha-\beta-1)\sum_{k=1}^{\infty}z^k \\
    z^2 q(z) &= -\alpha \beta \sum_{k=1}^{\infty}z^k
\end{align*}
in $|z|<1$.
For these equations we have the indicial polynomial 
\begin{equation*}
    p(s) = s(s-1)+sc
\end{equation*}
which always have the roots $s=0$ and $s=1-\gamma$.
The Frobenius solution has been fully studied and depends on the difference between the roots $\gamma -1$ which reduces to investigating each case when $\gamma$ is an integer or not. For our purposes, it suffices to know the following solutions around $z=0$:
\begin{enumerate}
    \item $\gamma$ is not an integer: $w(z) = A \, _2F_1(\alpha,\beta,\gamma;z) + B z^{1-\gamma} _2F_1(\alpha - \gamma +1,\beta - \gamma +1,2-\gamma;z)$
    \item $\gamma \in \mathbb{N}_0$ ($s_1 =s_2)$: $w(z) = c_1\, _2F_1(\alpha,\beta,1;z) + c_2\log(z)h(z) $
\end{enumerate}
for $h$ some analytic function which can be found in Abramowitz and Stegun \cite{abramowitz1965handbook}.
The hypergeometric function is defined as 
\begin{equation*}
    _2F_1(\alpha,\beta,\gamma;z) = \sum_{n=0}^{\infty}\frac{(a)_n(b)_n}{(c)_n}\frac{z^n}{n!}.
\end{equation*}
\subsection{Mode stability at the ODE blow-up profile}

To apply hypergeometric theory, write \(y=2z-1\). Then \eqref{p=1_hypergeometric} becomes
\begin{equation}
z(1-z)\,\varphi''+(\lambda-2\lambda z)\,\varphi'-\lambda(\lambda-1)\,\varphi=0,
\label{hypergeometric_ode_blow_up}
\end{equation}
i.e. the Gauss hypergeometric equation with parameters
\[
a(\lambda)=\lambda-1,\qquad b(\lambda)=\lambda,\qquad c(\lambda)=\lambda.
\]
The points \(z=0,1\) are regular singular, and Frobenius–Fuchs theory yields exponents \(\{0,\,1-\lambda\}\) at both \(z=0\) and \(z=1\).

\begin{proposition}\label{ODE_mode_stability_proposition}
If \(\mathrm{Re}\,\lambda>-1\), the only \(\varphi\in C^\infty[-1,1]\) solving \eqref{hypergeometric_ode_blow_up} occur for \(\lambda\in\{0,1\}\). Therefore the ODE blow up solution is mode stable.
\end{proposition}

\begin{proof}
A fundamental system near \(z=0\) is
\[
\varphi_1(z)={}_2F_1(\lambda-1,\lambda;\lambda;z)=(1-z)^{1-\lambda},\qquad
\varphi_2(z)=z^{1-\lambda}{}_2F_1(1,1;2-\lambda;z).
\]
Smoothness at \(z=0\) forces either \(1-\lambda\in\mathbb N_0\) or the coefficient of the singular branch \(\varphi_2\) to vanish. By symmetry, the same condition arises at \(z=1\). Thus \(1-\lambda=n\in\mathbb N_0\).
For these \(n\), both \(z^n\) and \((1-z)^n\) solve \eqref{hypergeometric_ode_blow_up} and are linearly independent with Wronskian
\(\mathcal W(z^n,(1-z)^n)=-(1-\lambda)\,z^{-\lambda}(1-z)^{-\lambda}\not\equiv 0\).
When \(\mathrm{Re}\,\lambda>-1\), the only possibilities are \(n=0,1\), i.e. \(\lambda\in\{1,0\}\). However, at \(\lambda=1\) we obtain the constant mode, and at \(\lambda=0\) the geometric eigenspace is two-dimensional, \(\mathrm{span}\{1,y\}\).
\end{proof}

In the original \(y\)-variable, representative eigenfunctions are
\[
\varphi(y)\in\mathrm{span}\{(1+y)^{1-\lambda},\ (1-y)^{1-\lambda}\}.
\]
\subsection{Mode stability of generalised self-similar solutions}
We transform the eigenequation \eqref{eigenequation} for general $p\in(0,1)$ into a hypergeometric ODE using the Lorentz transformation in similarity coordinates, first introduced in \cite{merle2007existence}.

\medskip
\noindent\textbf{Lorentz transform and similarity variables.}
Recall that under the Lorentz transformation with parameter $\gamma=\sqrt{1-p}$, the solutions
\[
u_{p,1,\kappa,T,x_0}(x,t)
=-p\log\!\Big(1-\frac{t}{T}\Big)-p\log\!\Big(1+\sqrt{1-p}\,\frac{x-x_0}{T-t}\Big)+\kappa
\]
correspond to
\begin{align*}
v_{p,1,\kappa,T,x_0}(x',t')
&=-(1-\gamma^2)\log(c_1+t')+c_2\\
&=-p\log\!\big(T-t'+\sqrt{1-p}\,(x'-x_0')\big)+p\log T+\kappa,
\end{align*}
which solves
\begin{equation}
v_{t't'}-v_{x'x'}=\frac{1}{p}\Big((v_{t'})^2-2\sqrt{1-p}\,v_{x'}v_{t'}+(1-p)(v_{x'})^2\Big).
\label{transformed}
\end{equation}
Thus $v_{p,1,\kappa,T,x_0}$ reduces to a spatially homogeneous ODE blow-up for \eqref{transformed}. However, the stability of $v_{p,1,\kappa,T,x_0}$ does not automatically imply the stability of $u_{p,1,\kappa,T,x_0}$, because:
\begin{enumerate}
\item perturbations for $u$ live on $\{t=0\}$, while perturbations for $v$ live on $\{t'=0\}=\{t-\gamma x=0\}$;
\item stability forward in $t'$ does not a priori imply stability forward in $t$.
\end{enumerate}
(As observed in \cite{liu2025note}, the advantage is discrete-spectral equivalence: one can study the spectrum at $u_{p,1,\kappa,T,x_0}$ by analysing the spectrum at $v_{p,1,\kappa,T,x_0}$ in self-similar variables.)

\medskip
\noindent\textbf{Eigenequation and coefficients.}
For $p\in(0,1)$ the separated eigenequation can be written in partial fractions as
\[
\varphi''+\Big(\frac{A}{y+1}+\frac{B}{y-1}+\frac{C}{1+y\sqrt{1-p}}\Big)\varphi'
+\Big(\frac{D}{y-1}+\frac{E}{y+1}+\frac{F}{1+y\sqrt{1-p}}\Big)\varphi=0,
\]
with $A = \lambda - \sqrt{1-p}$,
    $B = \lambda + \sqrt{1-p}$, 
    $C = 2\sqrt{1-p}$,
    $D  = \frac{\lambda^2}{2}-\frac{\lambda}{2}+\lambda\sqrt{1-p}$,
    $E = -\frac{\lambda^2}{2}+\frac{\lambda}{2}+\lambda\sqrt{1-p} $,
    $F = -2\lambda(1-p)$. \\
\medskip
\noindent\textbf{Similarity variables in Lorentzian coordinates.}
Define
\[
\tau'=\log T'-\log(T'-t'),\qquad y'=\frac{x'-x_0'}{T'-t'}.
\]
Then, for $V(\tau',y'):=v(x'(\tau',y'),t'(\tau',y'))$, the PDE \eqref{transformed} becomes
\begin{align}
&\partial_{\tau'\tau'}V+\partial_{\tau'}V+2y'\,\partial_{\tau'y'}V+2y'\,\partial_{y'}V+(y'^2-1)\,\partial_{y'y'}V \nonumber\\
&\qquad =\frac{1}{p}\Big( (\partial_{\tau'}V+y'\partial_{y'}V)^2
-2\sqrt{1-p}\,\partial_{y'}V\cdot(\partial_{\tau'}V+y'\partial_{y'}V)
+(1-p)(\partial_{y'}V)^2\Big).
\label{transformed_similarity_lorentz}
\end{align}
For our family $U_{p,1,\kappa}(\tau,y)=p\tau-p\log(1+y\sqrt{1-p})+\kappa$, the similarity-coordinate Lorentz map is
\begin{equation}
\tau=\tau'-\log\!\Big(\frac{1-\sqrt{1-p}\,y'}{p}\Big),\qquad
y=\frac{y'-\sqrt{1-p}}{1-\sqrt{1-p}\,y'},
\label{Lorentz_in_similarity_coordinates}
\end{equation}
and hence
\[
V_{p,1,\kappa}(\tau',y')=p\,\tau'+\kappa.
\]

\medskip
\noindent\textbf{Linearisation and eigenequation in Lorentz variables.}
Linearising \eqref{transformed_similarity_lorentz} about $V_{p,1,\kappa}$ yields, for $\xi$,
\begin{equation}
\partial_{\tau'\tau'}\xi-\partial_{\tau'}\xi+2y'\,\partial_{\tau'y'}\xi+2\sqrt{1-p}\,\partial_{y'}\xi+(y'^2-1)\,\partial_{y'y'}\xi=0.
\label{linearised_lorentz}
\end{equation}

\begin{proposition}\label{proposition_solving_linearisedeq}
If $U=U_{p,1,\kappa}+\eta$ with $\eta$ solving the linearised equation in $(\tau,y)$, then
\[
\xi(\tau',y'):=\eta\big(\tau(\tau',y'),\,y(\tau',y')\big)
\]
solves \eqref{linearised_lorentz} under the similarity-coordinate Lorentz map \eqref{Lorentz_in_similarity_coordinates}.
\end{proposition}

\begin{proof}
Direct substitution.
\end{proof}

With the ansatz $\xi(\tau',y')=e^{\lambda \tau'}\psi(y')$, \eqref{linearised_lorentz} reduces to
\begin{equation}
(y'^2-1)\psi''+(2\lambda y'+2\sqrt{1-p})\psi'+\lambda(\lambda-1)\psi=0.
\label{eigenLorentz_not_standard_form}
\end{equation}
In standard hypergeometric form, with $y'=2z'-1$ and $\psi(y')=\phi(z')$,
\begin{equation}
z'(1-z')\phi''+\big(\lambda-\sqrt{1-p}-2\lambda z'\big)\phi'-\lambda(\lambda-1)\phi=0,
\label{eigen-equation-lorentz}
\end{equation}
i.e. Gauss’s hypergeometric equation with parameters
\[
a=\lambda,\qquad b=\lambda-1,\qquad c=\lambda-\sqrt{1-p}.
\]
Thus the exponents at $z'=0$ and $z'=1$ are $\{0,\,1-\lambda+\sqrt{1-p}\}$ and $\{0,\,1-\lambda-\sqrt{1-p}\}$, respectively.

\medskip
\noindent\textbf{Local Frobenius analysis at $z'=1$.}
Then we have the following cases
\begin{itemize}
    \item Case 1: $s_1 = 1-\lambda - \sqrt{1-p}$ and $s_2 =0$ since $\mathfrak{Re}(\lambda) \leq 1-\sqrt{1-p}$
    \item  Case 2: $s_1=0$ and $s_2 = 1-\lambda-\sqrt{1-p}$ assuming that $\mathfrak{Re}(\lambda)>1-\sqrt{1-p}$
\end{itemize}

\noindent \textbf{In case 1}, we have the sub-case where $s_1-s_2 \notin \mathbb{N}_0$, then we have the independent solutions
\begin{align*}
    \phi_{1,1} &= (1-z')^{s_1}h_1(z'-1) \\
    \phi_{1,2} &= h_2(z'-1) 
\end{align*}
for analytic $h_1,h_2$. Since $s_1\notin \mathbb{N}$, any smooth solution must be a multiple of the analytic function $\phi_{0,2} = \, _2F_1(a,b,1+a+b-c)$.
\\

\noindent The next subcase is if $s_1-s_2 \in \mathbb{N}_0$ and so $1-\lambda -\sqrt{1-p}\in \{0,1,2,\dots\}$. By restricting ourselves to $\mathfrak{Re}(\lambda)\geq0$, and noting that $\sqrt{1-p} \in (0,1)$, then we have $s_1=1-\lambda -\sqrt{1-p}<1-\lambda\leq 1$ which implies $s_1=0$. The two independent solutions are thus
\begin{align*}
    \phi_{1,3} &= h_1(z'-1) \\
    \phi_{1,4} &= h_2(z'-1) + c_1\log(z'-1) \phi_{11}.
\end{align*}
Since $s_1-s_2=0$, $c_1\neq 0$, so any smooth solution must be a multiple of $\phi_{11} = h_1(z'-1) = \, _2F_1(a,b,1+a+b-c)$. \\
\noindent Finally for \textbf{case 2}, we have that $\mathfrak{Re}(\lambda)>1-\sqrt{1-p}$. The first subcase is if $s_1-s_2 \notin \mathbb{N}_0$ which implies $\mathfrak{Re}(s_2) <0$, then we have the independent solutions
\begin{align*}
    \phi_{2,1} &= h_1(z'-1) \\
    \phi_{2,2} &= (z'-1)^{s_2}h_2(z'-1) 
\end{align*}
Since $\mathfrak{Re}(s_2)<0$ then $\phi_{2,2}$ is not smooth at $z'=1$, so the local smooth solution must be a multiple of $h_1(z'-1) = \, \, _2F_1(a,b,1+a+b-c;1-z')$. The second sub-case is $s_1-s_2 \in \mathbb{N}_0$ for which we have the solutions 
\begin{align*}
    \phi_{2,3} &= h_1(z'-1) \\
    \phi_{2,4} &= (z'-1)^{s_2}h_2(z'-1) + c\log(z'-1)\phi_{2,1} 
\end{align*}
Since $\mathfrak{Re}(\lambda)>1-\sqrt{1-p}$, $s_2\neq 0$ the constant $c$ may be $0$, but $\mathfrak{Re}(s_2)<0$ and so $\phi_{2,4}$ is not smooth regardless of $c$, so the local smooth solution is again a multiple of $\phi_{2,3}=h_1(z'-1) = \,_2F_1(a,b,1+a+b-c) $

In all cases, any smooth local solution at $z'=1$ is a multiple of
\[
\phi_{+}(z')={}{}_2F_1(a,b;1+a+b-c;1-z').
\]
A symmetric Frobenius analysis at $z'=0$ yields the analogous conclusion there.

\begin{proposition}\label{mode_stability_generalised_solutions}
For $p\in(0,1)$ and $\operatorname{Re}\lambda\ge0$, there are no nonzero smooth solutions $\phi_\lambda\in C^\infty[0,1]$ to \eqref{eigen-equation-lorentz}, except for $\lambda\in\{0,1\}$. Therefore, the profiles $U_{p,1,\kappa}$ are mode stable.
\end{proposition}
\begin{proof}
    By the above analysis that for any $\lambda\in \mathbb{C}$ with $\mathfrak{Re}(\lambda)\geq 0$ and any non-zero solution $\phi \in C^\infty[0,1]$ solving \eqref{linearised_lorentz}, then $\phi$ must also be analytic on $[0,1]$. To show the mode stability, we show that if $\lambda \notin \{0,1\}$ and $\mathfrak{Re}(\lambda)\geq 0$ then $\phi$ fails to be analytic on $[0,1]$, which proves the claim by the contrapositive. By the above analysis, we found that the smooth local solution about $z'=1$ is a multiple of
\begin{equation*}
    \phi(z') = \, _2F_1(a,b,1+a+b-c;1-z')
\end{equation*}
which is an analytic function around $z'=1$. However, the hypergeometric function is analytic everywhere except possibly at the singular points. Thus, when the radius of convergence is precisely 1, then there must be singularity located at $z'=0$, rendering the function non-analytic on $[0,1]$. By definition of the hypergeometric function, the coefficients $\alpha_n$ of $_2F_1(a,b,1+a+b-c;1-z')$ are
\begin{equation*}
    \alpha_n(\lambda) = \frac{(a)_n(b)_n}{n!(1+a+b-c)_n} =  \frac{(\lambda)_n(\lambda-1)_n}{n!(\lambda+\sqrt{1-p})_n}.
\end{equation*}
If $\lambda \notin \{0,1\}$ as well as $\mathfrak{Re}(\lambda)\geq0$ then $\alpha_n \neq 0$ for any $n\in \mathbb{N}$. Then the limiting behaviour of the radius of convergence may be found as
\begin{equation*}
    r_n(\lambda) = \frac{a_{n+1}(\lambda)}{a_n(\lambda)} = \frac{(\lambda+n)(\lambda+n-1)}{(n+1)(\lambda+\sqrt{1-p}+n)} \to 1, \,\, \, n\to \infty.
\end{equation*}
\end{proof}
\section{Spectral analysis of \texorpdfstring{$\widetilde{\mathbf L}_p$}{Lp}}
Recall from \eqref{first-order system} that the linearised flow about $U_{p,1,\kappa}$ can be written, for $\tau\in[0,\infty)$ and $y\in[-1,1]$, as
\begin{equation}
\partial_\tau
\begin{pmatrix} q_1 \\ q_2 \end{pmatrix}
=\widetilde{\mathbf L}_p
\begin{pmatrix} q_1 \\ q_2 \end{pmatrix},\qquad
\widetilde{\mathbf L}_p
\begin{pmatrix} q_1 \\ q_2 \end{pmatrix}
:=
\begin{pmatrix}
 -y\,\partial_y q_1 + q_2 \\
 \partial_{yy} q_1 - y\,\partial_y q_2 + \Big(\dfrac{2p}{1+y\sqrt{1-p}} - 1\Big) q_2
\end{pmatrix}.
\end{equation}

Following \cite{ghoul2025blow}, we add and subtract a boundary trace to the free wave part:
\begin{equation*}
\widetilde{\mathbf L}_p=\widetilde{\mathbf L}+\mathbf L_{p,1},\qquad
\widetilde{\mathbf L}
\begin{pmatrix} q_1 \\ q_2 \end{pmatrix}
:=
\begin{pmatrix}
 -y\,\partial_y q_1 + q_2 - q_1(-1) \\
 \partial_{yy} q_1 - y\,\partial_y q_2 - q_2
\end{pmatrix},\qquad
\mathbf L_{p,1}
\begin{pmatrix} q_1 \\ q_2 \end{pmatrix}
:=
\begin{pmatrix}
 q_1(-1) \\
 \dfrac{2p}{1+y\sqrt{1-p}}\,q_2
\end{pmatrix}.
\end{equation*}
Here the map $q_1\mapsto q_1(-1)$ is rank-one (hence compact) on our state space. Secondly, although the multiplier $y\mapsto \dfrac{2p}{1+y\sqrt{1-p}}$ is smooth on $[-1,1]$ for $p\in(0,1)$, by construction, the multiplication operator $q_2 \mapsto \frac{2p}{1+y\sqrt{1-p}}q_2$ is not compact.

\subsection{Functional framework}
We use the same framework as in Ostermann \cite{ostermann2024stable} adapted to 1D. For smooth $\mathbf q=(q_1,q_2)$ and $\widetilde{\mathbf q}=(\widetilde q_1,\widetilde q_2)$ set
\[
\langle \mathbf q,\widetilde{\mathbf q}\rangle_0
:=\int_{-1}^1 (\partial_y q_1)\,\overline{\partial_y \widetilde q_1}\,dy
+\int_{-1}^1 q_2\,\overline{\widetilde q_2}\,dy
+ q_1(-1)\,\overline{\widetilde q_1(-1)},
\]
and, for $k\ge1$,
\[
\langle \mathbf q,\widetilde{\mathbf q}\rangle_k
:=\int_{-1}^1 \partial_y^{k+1}q_1\,\overline{\partial_y^{k+1}\widetilde q_1}\,dy
+\int_{-1}^1 \partial_y^{k}q_2\,\overline{\partial_y^{k}\widetilde q_2}\,dy
+\langle \mathbf q,\widetilde{\mathbf q}\rangle_0,\qquad
\|\mathbf q\|_k^2:=\langle \mathbf q,\mathbf q\rangle_k.
\]
By the trace theorem, $q_1(-1)$ is well-defined for $q_1\in H^1(-1,1)$, and for all $k\ge0$,
\begin{equation}\label{eq:norm-equiv}
\|\mathbf q\|_k \simeq \|q_1\|_{H^{k+1}(-1,1)}+\|q_2\|_{H^{k}(-1,1)}.
\end{equation}
We write $\mathcal H^k:=H^{k+1}(-1,1)\times H^{k}(-1,1)$; the completion of $C^{k+1}[-1,1]\times C^k[-1,1]$ under $\|\cdot\|_k$ is $\mathcal H^k$.

\subsection{Generation of the semigroup}
All proofs follow \cite{ghoul2025blow} (and standard semigroup theory \cite{engel2000one}); we only state the results we need.

\begin{lemma}[Dissipativity of the modified free wave]\label{lem:diss}
For $k\geq 0$, we have 
\[
\mathfrak{Re}\,\langle \widetilde{\mathbf L}\mathbf q,\mathbf q\rangle_k \le -\tfrac12\,\|\mathbf q\|_{\mathcal H^k}^2
\quad\text{for all} \mathbf \quad \mathbf{q} \in (C^\infty[-1,1])^2.
\]
In particular, $\widetilde{\mathbf L}+\tfrac12$ is dissipative. 
\end{lemma}
Next we have that the range of $\lambda -\widetilde{\mathbf{L}}$ is dense in $\mathcal H^k$, for some $\lambda>-\frac{1}{2}$. $\lambda = 0$ is chosen for convenience. 
\begin{lemma}[Dense range at \texorpdfstring{$\lambda=0$}{lambda=0}]
    For any $k\geq0$, $\varepsilon>0$ and any $\mathbf{g}\in \mathcal{H}^k$, there exists $\mathbf{q}\in C^\infty[-1,1]\times C^\infty[-1,1]$ such that 
    \begin{equation*}
        \|-\widetilde{\mathbf{L}}\mathbf{q}-\mathbf{g}\|_k < \varepsilon
    \end{equation*}
\label{lemma_density_surjective}
\end{lemma}
By the Lumer Phillips theorem, we now have that the closure of the modified wave operator with domain $\mathcal{H}^k$  generates a $C_0$ semigroup. 
\begin{proposition}[Generation for $\widetilde{\mathbf L}$]\label{prop:gen-free}
The closure of $\widetilde{\mathbf L}$ generates a contraction $C_0$-semigroup $(T(\tau))_{\tau\ge0}$ on $\mathcal H^k$ for $\widetilde{\mathbf L}+\tfrac12$. Consequently,
\[
S(\tau):=e^{-\frac{\tau}{2}}T(\tau)
\]
is the $C_0$-semigroup generated by $\widetilde{\mathbf L}$ and satisfies
\(
\|S(\tau)\|_{\mathcal L(\mathcal H^k)}\le e^{-\frac{\tau}{2}}.
\)
\end{proposition}
To pass to the full operator we now treat $\mathbf{L}_{p,1}$. 
\begin{proposition}[Boundedness of $\mathbf L_{p,1}$ and generation for $\widetilde{\mathbf L}_p$]
    For all $k\geq 0$, and all $p\in(0,1)$, the operator $\mathbf{L}_{p,1}$ is bounded on $\mathcal{H}^k$. As a corollary $\mathbf{L}_{p}\coloneqq \mathbf{L}+\mathbf{L}_{p,1}$ generates a strongly continuous one-parameter semi-group $\mathbf{S}_p(\cdot) : [0,\infty)\to \mathcal{L}(\mathcal{H}^k)$ which satisfies 
    \begin{equation}
        \|\mathbf{S}_p(\tau)\|_{\mathcal{L}(\mathcal{H}^k)} \leq Me^{\left(-\frac{1}{2}+\|\mathbf{L}_{p,1}\|_{\mathcal{L}(\mathcal{H}^k)}\right)\tau}
    \label{growth_bound_linearsed_operator}
    \end{equation}
    \label{L_p semigroup generation}
    for all $\tau \geq 0$ and $M>0$ is a constant independent of $p$. 
\end{proposition}
\begin{proof}
    To show $\mathbf{L}_{p,1}$ is bounded on $\mathcal{H}^k$, take $q\in \mathcal{H}^k$
    \begin{align*}
        \|\mathbf{L}_{p,1} \mathbf{q}\|_{\mathcal{H}^k} = \|q_1(-1)\|_{H^{k+1}}+\left\|\frac{2p}{1+y\sqrt{1-p}}q_2\right\|_{H^k} \lesssim 
        \|q_1\|_{L^\infty} + 4\|q_2\|_{H^k} \lesssim \|\mathbf{q}\|_{\mathcal{H}^k}
    \end{align*}
    Then using the bounded perturbation theorem (\cite{engel2000one}~p.158), the sum of the operators $\mathbf{L}_p = \mathbf{L}+ \mathbf{L}_{p,1}$ generates the $C_0$ semigroup with the growth bound \eqref{growth_bound_linearsed_operator}. 
\end{proof}
\subsection{\texorpdfstring{A New Decomposition of $\mathbf{L}_p$}{A Decomposition of Lp}}
Recall 
\begin{equation*}
    \mathbf{L}_p = \begin{pmatrix}
        -y\partial_y & 1 \\
        \partial_{yy} & -y\partial_y + \left(U_p(y)-1 \right)  
    \end{pmatrix} = \mathbf{L}+\mathbf{L}_{p,1}
\end{equation*}
where we denote $U_p(y) = \frac{2p}{1+y\sqrt{1-p}} =  2(\partial_{\tau}U_{p,1,k}+y\partial_y U_{p,1,\kappa})$. As in the $u_{tt}-u_{xx}=(u_{x})^2$ problem, the perturbation $\mathbf{L}_{p,1}$ is merely bounded, and neither compact nor $\mathbf{L}$-compact. This complicates finding information about the spectrum of $\mathbf{L}_p$. Indeed, recall 
\begin{equation*}
    \mathbf{L}_{p,1}\mathbf{q} = \begin{pmatrix}
        q_1(-1) \\ 
        \frac{2p}{1+y\sqrt{1-p}}q_2
    \end{pmatrix}.
\end{equation*}
Although $q_1(-1)$ is a compact operator the second component, a multiplication operator, is not compact. In addition, $\mathbf{L}_{p,1}$ is not $\mathbf{L}$-compact. Indeed, suppose $\mathbf{q}_n = (q_{1n},q_{2n}),\mathbf{Lq}_n \in H^{k+1}\times H^k$ such that, $\|\mathbf{q}_n\|_{\mathcal{H}^k}\leq C_1$, and $\|\mathbf{Lq}_n\|_{\mathcal{H}^k}\leq C_2$. It does not follow that $U_p(y)q_{2n}$ contains a convergent subsequence in $H^{k}$. Assuming the bound $\|\mathbf{Lq}_n\|_{\mathcal{H}^k}$ only allows us to deduce that $(1-y^2)\partial_y q_2 \in H^k(-1,1)$, rather than $\partial_y q_2 \in H^k(-1,1)$ (which would let us extract a convergent subsequence by the compact embedding $H^{k+1}(-1,1) \ssubset H^k(-1,1)$).

As a step towards finding information on the spectrum, we therefore investigate the  dissipation of the overall operator $\mathbf{L}_p$, firstly in the base space $H^1 \times L^2$. By Lemma \ref{lem:diss}
\begin{align*}
    \mathfrak{Re}\langle \mathbf{L}_p \mathbf{q},\mathbf{q}\rangle_{0} &= \mathfrak{Re}\langle \mathbf{Lq} + \mathbf{L}_{p,1}\mathbf{q},\mathbf{q}\rangle_{H^1\times L^2} \\
    &\leq -\frac{1}{2}\|\mathbf{q}\|_0^2 + \mathfrak{Re}\int_{-1}^{1}\frac{2p}{1+y\sqrt{1-p}}|q_2(y)|^2 \,dy \\
    &\leq -\frac{1}{2}(\|q_1\|_{H^1}^2+\|q_2\|_{L^2}^2) + 2(1+\sqrt{1-p})\|q_2\|_{L^2}^2 \\
    &= -\frac{1}{2}\|q_1\|_{H^1}^2 + \left(\frac{3}{2}+2\sqrt{1-p}\right)\|q_2\|_{L^2}^2
\end{align*}
we see that dissipativity is not guaranteed for $\mathcal{H}^0$. We therefore look in higher order spaces. To proceed with computations in higher order Sobolev spaces, we first need a technical lemma showing how the operator $\mathbf{L}_p$ commutes with derivatives $\partial_y^k$. 
\begin{lemma}
    Given $p\in (0,1)$, and $k\geq 0$, we have 
    \begin{equation*}
        \partial_y^k\mathbf{L}_p = \mathbf{L}_{p,k} \partial_y^k + \mathbf{L}_{p,k}'
    \end{equation*}
    where 
    \begin{equation*}
        \mathbf{L}_{p,k} = \begin{pmatrix}
            -y\partial_y -k & 1 \\
            \partial_{yy} & -y\partial_y - k - 1 + U_p(y)
        \end{pmatrix}
    \end{equation*}
    and 
    \begin{equation*}
        \mathbf{L}_{p,k}' = \begin{pmatrix}
            0 & 0 \\
            0 & [\partial_y^k,U_p(y)]
        \end{pmatrix}
    \end{equation*}
    which satisfies the pointwise bound 
    \begin{equation}
        |\mathbf{L}_{p,k}'\mathbf{q}| \lesssim_{p} \begin{pmatrix}
            0\\
            \sum_{j=0}^{k-1}|\partial_y^j q_2|
        \end{pmatrix}
        \label{pointwise_bound_for_L_pk'}
    \end{equation}
    \label{Lemma_commutation}
\end{lemma}
\begin{proof}
    To compute the commutation with diagonal terms, we use that 
    \begin{equation*}
        [\partial_y^k,y\partial_y] = k\partial_y^k.
    \end{equation*}
    For the pointwise bound, we compute the non-zero component of $\mathbf{L}_{p,k}'\mathbf{q}$.
    \begin{align*}
        [\partial_y^k, U_p(y)] q_2 
        &= \partial_y^k\left(U_p(y) q_2\right) - U_p(y) \partial_y^k q_2 \\
        &= \sum_{j=0}^{k} \binom{k}{j} \left(\partial_y^{k-j} U_p(y)\right) \left(\partial_y^j q_2\right) - U_p(y) \partial_y^k q_2 \\
        &= \sum_{j=0}^{k-1} \binom{k}{j} \left(\partial_y^{k-j} U_p(y)\right) \left(\partial_y^j q_2\right)
    \end{align*}
    Note that \begin{equation*}
        \partial_y^{k-j}\left(\frac{2p}{1+y\sqrt{1-p}}\right) = (-1)^{k-j}(k-j)!\left(\sqrt{1-p}\right)^{k-j}\frac{2p}{(1+y\sqrt{1-p})^{k-j+1}}
    \end{equation*}
    so 
    \begin{equation*}
        \left \|\partial_y^{k-j}U_p(y)\right\|_{L^\infty(-1,1)} = \left \|(k-j)!\frac{2p\left(\sqrt{1-p}\right)^{k-j}}{\left(1+y\sqrt{1-p}\right)^{k-j+1}}\right\|_{L^\infty(-1,1)} \leq k!\frac{2p\left(\sqrt{1-p}\right)^{k-j}}{\left(1-\sqrt{1-p}\right)^{k-j+1}}
    \end{equation*}
    which we take as the $p$-dependent constant in the pointwise bound \eqref{pointwise_bound_for_L_pk'}.
\end{proof}
We now define an energy which is an equivalent inner product to $H^k$, isolating the  highest order derivative. 
\begin{definition}
    For $k\geq 0$, we define the new inner product $\langle \langle \cdot, \cdot \rangle \rangle$ by 
    \begin{equation*}
        \langle \langle \Psi,\widetilde{\Psi} \rangle \rangle_k \coloneqq (\partial_y^k \Psi, \partial_y^k \widetilde{\Psi})_{L^2(-1,1)} + (\Psi, \widetilde{\Psi})_{L^2(-1,1)} \equiv \int_{-1}^{1}\partial_y^k \Psi \overline{\partial_y^k \widetilde{\Psi}}\,dy + \int_{-1}^{1}\Psi \overline{\widetilde{\Psi}}\,dy
    \end{equation*}
    For tuples $\mathbf{q}=(q_1,q_2)\in H^{k+1}(-1,1)\times H^k(-1,1)$, we use 
    \begin{equation*}
        \langle \langle \mathbf{q},\widetilde{\mathbf{q}} \rangle\rangle_k \coloneqq \langle \langle q_1, \widetilde{q}_1 \rangle \rangle_{k+1} + \langle \langle q_2, \widetilde{q}_2\rangle \rangle_{k}
    \end{equation*}
\end{definition}
The induced norm $\||\Psi\||_k \coloneqq \langle\langle \Psi,\Psi \rangle \rangle $ is equivalent to the standard $H^k$ norm (See \cite{adams2003sobolev},p.135).

Now, in order to control $\mathbf{L}_{p,k}'$, we need the following coercivity estimate. 
\begin{lemma}[Subcoercivity]
    Given $m\geq1$, there exists a sequence $\epsilon_n >0$, with $\lim_{n\to \infty}\epsilon_n = 0$, and $\{\Pi_i\}_{1\leq i\leq n} \subset H^{m}(-1,1)$, and constants $c_n>0$ such that for all $n\geq 0$ and $\Psi\in H^{m}(-1,1)$,
    \begin{equation}
        \epsilon_n \langle \langle \Psi, \Psi \rangle \rangle_m \geq \|\Psi\|_m^2 -c_n\sum_{i=1}^{n}(\Psi,\Pi_i)_{L^2(-1,1)}^2
    \label{subcoercivity_eq}
    \end{equation}
\label{Subcoercivity}
\end{lemma}
\begin{proof}
    A proof can be found in \cite{merle2022blow,kim2022self,ghoul2025blow}.
\end{proof} 
Thanks to the above Lemma, we are able to decompose the linearised operator into a sum of a maximally dissipative operator plus a finite rank projection. 
\begin{proposition}[Maximal Dissipativity]
    Let $p\in(0,1)$, $k\geq 4$, for $\varepsilon \in (0,\frac{1}{2})$, there exists an integer $N=N(\varepsilon,k,p)$ and vectors  $\left(\mathbf{\Pi}_{i,p}\right)_{1\leq i\leq N}\subset \mathcal{H}^k$ such that for the finite rank projection operator 
    \begin{equation}
        \widehat{\mathbf{P}}_p = \sum_{i=1}^{N}\langle \langle \cdot, \mathbf{\Pi}_{p,i}\rangle \rangle_k \mathbf{\Pi}_{p,i}
        \label{finite_rank_projector}
    \end{equation}
    the modified operator 
    \begin{equation*}
        \widehat{\mathbf{L}}_p \coloneqq \mathbf{L}_p - \widehat{\mathbf{P}}_p
    \end{equation*}
    satisfies: \\

(Dissipativity):
    \begin{equation}
\begin{aligned}
    \forall \mathbf{q}\in \mathcal{D}(\mathbf{L}_p), \quad 
    \mathfrak{Re}\langle \langle \widehat{\mathbf{L}}_p \mathbf{q},\mathbf{q}\rangle \rangle_k 
    &\leq -\left(k-\frac{7}{2}-\varepsilon \right)\langle \langle \mathbf{q},\mathbf{q}\rangle \rangle_k \\
    &\leq \left(\varepsilon-\frac{1}{2}\right)\langle \langle \mathbf{q},\mathbf{q}\rangle \rangle_k
\end{aligned}
\label{dissipative_estimate}
\end{equation}\\

(Maximality)
\begin{equation}
    \lambda - \widehat{\mathbf{L}}_p \quad \text{is surjective for some}\quad  \lambda>0.
\label{maximality}
\end{equation}
\label{Maximal dissipative prop}
\end{proposition}
\begin{proof}
    For $\mathbf{q}\in (C^\infty[-1,1])^2$, we have using Lemma \ref{Lemma_commutation}, that 
    \begin{align*}
        \langle \langle -\mathbf{L}_p \mathbf{q},\mathbf{q}\rangle \rangle_k &= (-\mathbf{L}_p\mathbf{q},\mathbf{q})_{\dot H^{k+1}\times \dot H^k} + (-\mathbf{L}_p\mathbf{q},\mathbf{q})_{L^2\times L^2} \\
        &= (-(\partial^{k+1}\mathbf{L}_p \mathbf{q})_1,\partial^{k+1}q_1)_{L^2} + ( -(\partial^k \mathbf{L}_p \mathbf{q})_2,\partial^k q_2)_{L^2} + (-\mathbf{L}_p \mathbf{q},\mathbf{q})_{L^2\times L^2}\\
        &= (-(\mathbf{L}_{p,k+1}\partial^{k+1}\mathbf{q})_1,\partial^{k+1}q_1)_{L^2} + (-(\mathbf{L}_{p,k}\partial^k \mathbf{q})_2,\partial^k q_2)_{L^2} \\
        &+ (-(\mathbf{L}_{p,k+1}'\mathbf{q})_1,\partial^{k+1}q_1)_{L^2} + (-(\mathbf{L}_{p,k}'\mathbf{q})_2,\partial^k q_2)_{L^2} + (-\mathbf{L}_p\mathbf{q},\mathbf{q})_{L^2\times L^2}.
    \end{align*}
    For the first two terms of the above expression, integration by parts gives us 
    \begin{align*}
        &\mathfrak{Re}(-(\mathbf{L}_{p,k+1}\partial^{k+1}\mathbf{q})_1,\partial^{k+1}q_1)_{L^2} + \mathfrak{Re}(-(\mathbf{L}_{p,k}\partial^k \mathbf{q})_2,\partial^k q_2)_{L^2}\\
        &= \left(k +\frac{1}{2}\right)(\|\partial^{k+1}q_1\|_{L^2}^2 + \|\partial^{k}q_2\|_{L^2}^2) - \left(U_p(y)\partial^k q_2,\partial^k q_2\right)_{L^2} + \frac{1}{2}\left[(\partial^{k+1}q_1(\pm 1)) \mp (\partial^k q_2(\pm 1))\right]^2\\
        &\geq \left(k +\frac{1}{2}\right)(\|\partial^{k+1}q_1\|_{L^2}^2 + \|\partial^{k}q_2\|_{L^2}^2) -2(1+\sqrt{1-p})\|\partial^{k}q_2\|_{L^2}^2 \\
        &\geq \left(k+\frac{1}{2} \right)\|\partial^{k+1}q_1\|_{L^2}^2 + \left(k-\frac{7}{2}\right)\|\partial^k q_2\|_{L^2}^2.
    \end{align*}
    As for the last three terms, by Youngs inequality we have that for any $\varepsilon>0$
    \begin{align*}
        &|(-(\mathbf{L}_{p,k+1}'\mathbf{q})_1,\partial^{k+1}q_1)_{L^2}| + |(-(\mathbf{L}_{p,k}'\mathbf{q})_2,\partial^k q_2)_{L^2}| + |(-\mathbf{L}_p\mathbf{q},\mathbf{q})_{L^2\times L^2}| \\
        &= |(-(\mathbf{L}_{p,k}'\mathbf{q})_2,\partial^k q_2)_{L^2}| + |(\mathbf{L}_p\mathbf{q},\mathbf{q})_{L^2\times L^2}|\\
        &\leq \|(\mathbf{L}_{p,k}'\mathbf{q})_2\|_{L^2}\| \partial^k q_2\|_{L^2} + |(\mathbf{L}_p\mathbf{q},\mathbf{q})_{L^2\times L^2}| \\
        &\leq \frac{\varepsilon}{2}\|\partial^k q_2\|_{L^2}^2 + c_{\varepsilon,p,k}\left(\|q_1\|_{H^k}^2 + \|q_2\|_{H^{k-1}}^2\right).
    \end{align*}
    Therefore  
    \begin{align*}
         \mathfrak{Re}\langle \langle -\mathbf{L}_p \mathbf{q},\mathbf{q}\rangle \rangle_k &\geq \left(k+\frac{1}{2} \right)\|\partial^{k+1}q_1\|_{L^2}^2 + \left(k-\frac{7}{2}-\frac{\varepsilon}{2}\right)\|\partial^k q_2\|_{L^2}^2 - c_{\varepsilon,p,k}\|\mathbf{q}\|_{\mathcal{H}^{k-1}}^2 \\
         &\geq \left(k-\frac{7}{2}-\frac{\varepsilon}{2}\right)\left(\|\partial^{k+1}q_1\|_{L^2}^2 + \|\partial^k q_2\|_{L^2}^2\right) - c_{\varepsilon,p,k}\|\mathbf{q}\|_{\mathcal{H}^{k-1}}^2.
    \end{align*}
    Applying the subcoercivity estimate for $q_1$,$q_2$ respectively allows us to trade the lower derivative terms into a $\frac{\varepsilon}{2}$ multiple of the first factor, modulo a finite rank defect, i.e
    \begin{equation*}
      \langle \langle -\mathbf{L}_p \mathbf{q},\mathbf{q}\rangle \rangle_k \geq (k-\frac{7}{2}-\varepsilon)\langle \langle \mathbf{q},\mathbf{q}\rangle \rangle_k  - c_{\varepsilon,k,p}'\left(\sum_{i=1}^{N_1}(q_1, \Pi_i^{(1)})_{L^2}^2 + \sum_{i=1}^{N_2}(q_2, \Pi_i^{(2)})_{L^2}^2\right).
 \end{equation*}
  Now consider the linear functionals given by 
\begin{equation*}
    \mathbf{q}\to \sqrt{c_{\varepsilon,k,p}'}\left(q_1,\Pi_i^{(1)}\right)_{L^2} \quad \text{and} \quad \mathbf{q} \to \sqrt{c_{\varepsilon,k,p}'}\left(q_2,\Pi_i^{(2)}\right)_{L^2}
\end{equation*}
which are actually continuous on $\mathcal{H}^k$ by Cauchy Schwarz. Thus by the Riesz Representation Theorem there exists $(\mathbf{\Pi}_{p,i})_{1\leq i \leq N_1+N_2}\subset \mathcal{H}^k$ such that 
\begin{align*}
    \langle \langle \mathbf{q},\mathbf{\Pi}_{p,i} \rangle \rangle_k &= \sqrt{c_{\varepsilon,k,p}'}\left(q_1,\Pi_i^{(1)}\right)_{L^2} \quad \text{for} \quad 1\leq i \leq N_1 \\
    \langle \langle \mathbf{q}, \mathbf{\Pi}_{p,i}\rangle \rangle_k &= \sqrt{c_{\varepsilon,k,p}'}\left(q_1,\Pi_i^{(2)}\right)_{L^2} \quad \text{for} \quad N_1 +1 \leq i \leq N_1 + N_2 \coloneqq N.
\end{align*}
By defining the projection operator $\widehat{\mathbf{P}}_p$ as in \eqref{finite_rank_projector} and $\widehat{\mathbf{L}}_p \coloneqq \mathbf{L}_p - \widehat{\mathbf{P}}_p$, we clearly get the dissipativity of $\widehat{\mathbf{L}}_p$ when $\mathbf{q}\in \left(C^\infty[-1,1]\right)^2$. 

However, we must prove \eqref{dissipative_estimate} for $\mathbf{q}\in \mathcal{D}(\mathbf{L}_p) = \mathcal{H}^k$. Indeed, since $\mathbf{L}_p$ generates a $C_0$ semigroup on $\mathcal{H}^k$ by Proposition \ref{L_p semigroup generation}, we have the well-defined characterisation $\mathbf{L}_p \mathbf{q}=\lim_{\tau \to 0}\frac{(S_p(\tau)-I)\mathbf{q}}{\tau}  \in \mathcal{H}^k$ for arbitrary $\mathbf{q}\in \mathcal{H}^k$. Let us fix $\mathbf{q}\in \mathcal{H}^k$. Since $\mathbf{L}_p = \mathbf{L}+\mathbf{L}_{p,1}$, where the latter operator is bounded on $\mathcal{H}^k$, we have that $\mathbf{Lq} \in \mathcal{H}^k$. By the density of $\left(C^\infty[-1,1]\right)^2$ in $\mathcal{H}^k$, there is a sequence $(\mathbf{f}_n)_{n\in \mathbb{N}}\subset \left(C^\infty[-1,1]\right)^2$ with $\|\mathbf{f}_n -(-\mathbf{Lq})\|_{k}\to 0$. By the constructive proof of Lemma \ref{lemma_density_surjective}, for every $n\in \mathbb{N}$, there exists a unique $\mathbf{q}_n \in \left(C^\infty[-1,1]\right)^2 $ such that $-\mathbf{Lq}_n = \mathbf{f}_n$. We now show that $\mathbf{q}_n \to \mathbf{q}$. Indeed, by dissipativity of the free wave operator in Lemma \ref{lem:diss} and Cauchy Schwarz, we have 
\begin{align*}
    -\frac{1}{2}\|\mathbf{q}_n -\mathbf{q}_m\|_k^2 &\geq \langle \mathbf{L}(\mathbf{q}_n-\mathbf{q}_m),\mathbf{q}_n-\mathbf{q}_m \rangle_k \geq -\|\mathbf{L}(\mathbf{q}_n-\mathbf{q}_m)\|_k \|\mathbf{q}_n - \mathbf{q}_m \|_k \\
    &\implies \|\mathbf{q}_n -\mathbf{q}_m\|_k^2 \leq 2\|\mathbf{L}(\mathbf{q}_n-\mathbf{q}_m)\|_k \to 0.
\end{align*}
Thus $(\mathbf{q}_n)$ is Cauchy and so there exists $\mathbf{q}' \in \mathcal{H}^k$ with $\|\mathbf{q}_n - \mathbf{q}'\|_k \to 0$. Since the operator $\mathbf{L}$ is closed, we have $-\mathbf{Lq}=-\mathbf{Lq}'$. By Proposition \ref{prop:gen-free}, and the Hille-Yosida Theorem (see \cite{engel2000one}), $(-\frac{1}{2},\infty)\subset \rho(\mathbf{L})$. In particular $0\in \rho(\mathbf{L})$, meaning $-\mathbf{L}$ is invertible and $\mathbf{q}=-\mathbf{q}'$ and we conclude $\mathbf{q}_n \to \mathbf{q}'=\mathbf{q}$. Now, since $\widehat{\mathbf{L}}_p = \mathbf{L}+\mathbf{L}_{p,1}-\widehat{\mathbf{P}}_p$, where $\mathbf{L}_{p,1}$ and $\widehat{\mathbf{P}}_p$ are bounded on $\mathcal{H}^k$, we also have that $\widehat{\mathbf{L}}_p\mathbf{q}_n \to \widehat{\mathbf{L}}_p\mathbf{q}$. Therefore taking $n\to \infty$, \eqref{dissipative_estimate} holds for $\mathbf{q}\in \mathcal{D}(\mathbf{L}_p)$.

We finally show that $\widehat{\mathbf{L}}_p$ is maximal, i.e $\lambda - \widehat{\mathbf{L}}_p:\mathcal{H}^k \to \mathcal{H}^k$ is surjective for all $\lambda>0$. Since we have shown dissipativity of $\widehat{\mathbf{L}}_p$, it suffices to show this for one such $\lambda>0$. Because $\mathbf{L}_p$ was the generator of a $C_0$ semi-group with growth bound $\omega = -\frac{1}{2}+\|\mathbf{L}_{p,1}\|_{\mathcal{L}(\mathcal{H}^k)} < \infty$, we can always find $\mathfrak{Re}\lambda>\omega$ with $\mathfrak{Re}\lambda \in \rho(\mathbf{L}_p)$ and so $\lambda-\mathbf{L}_p$ invertible. Again by the Hille-Yosida, we have 
\begin{equation*}
    \|\left(\lambda - \mathbf{L}_p\right)^{-1}\|_{\mathcal{L}(\mathcal{H}^k)} \lesssim \frac{1}{(\lambda-\omega)} \lesssim \frac{2}{\lambda}.
\end{equation*}
Then writing 
\begin{equation*}
    (\lambda - \widehat{\mathbf{L}}_p) = \lambda - \mathbf{L}_p + \widehat{\mathbf{P}}_p = (\lambda - \mathbf{L}_p)(\mathbf{I} + (\lambda - \mathbf{L}_p)^{-1}\widehat{\mathbf{P}}_p)
\end{equation*}
we have by the boundedness of $\widehat{\mathbf{P}}_p$ and for $\lambda>0$ sufficiently large 
\begin{equation*}
    \|(\lambda - \mathbf{L}_p)^{-1}\widehat{\mathbf{P}}_p)\|_{\mathcal{L}(\mathcal{H}^k)} <1
\end{equation*}
so $(\mathbf{I} + (\lambda - \mathbf{L}_p)^{-1}\widehat{\mathbf{P}}_p)^{-1}$ is well-defined by a Neumann series. Therefore the above product is also invertible, and in particular surjective. 
\end{proof}
\begin{remark}[Choice of regularity]
    The Sobolev regularity $k\geq 4$ could be chosen larger to improve the dissipative estimate \eqref{dissipative_estimate}, thereby pushing the spectrum to the left. However we choose the smallest $k\in \mathbb{N}$ that ensures dissipativity. By comparison, the $(u_x)^2$ problem allows similar estimates for the solutions of wider parameter range $\alpha>0$, and consequently $k_{\alpha}$ depends on $\alpha$. Here we merely have parameter $p\in (0,1)$ so the smallest $k$ does not depend on $p$. 
\end{remark}
\subsubsection{The spectrum of the linearised operator}
We can now give a sufficiently detailed description of the spectrum of $\mathbf{L}_p$. 
\begin{proposition}
    Take $p\in (0,1)$, $k\geq 4$, then $\exists \omega_0 \in (0,\frac{1}{2})$ such that 
    \begin{equation*}
        \sigma(\mathbf{L}_p) \subset \{z\in \mathbb{C}: \mathfrak{Re}z \leq -\omega_0 \} \cup \{0,1\}
    \end{equation*}
    and $\{0,1\}\subset \sigma_p(\mathbf{L}_p)$. Furthermore, the geometric eigenspaces of eigenvalues $0$ and $1$ are spanned by the functions $\mathbf{f}_{0,p}$ and $\mathbf{f}_{1,p}$ respectively, where 
    \begin{equation*}
        \mathbf{f}_{0,p}(y) = \begin{pmatrix}
            1\\ 
            0 
        \end{pmatrix} \quad \text{and} \quad \mathbf{f}_{1,p}(y) = \begin{pmatrix}
            \frac{-p\sqrt{1-p}}{1+y\sqrt{1-p}} \\
            \frac{-p\sqrt{1-p}}{(1+y\sqrt{1-p})^2}
        \end{pmatrix}.
    \end{equation*}
    Moreover, there is a generalised eigenfunction $\mathbf{g}_{0,p}$ of the eigenvalue $0$ given by 
    \begin{equation*}
        \mathbf{g}_{0,p}(y) = \begin{pmatrix}
            -\log\left(1+y\sqrt{1-p}\right) -\frac{p}{2(1-p)}\frac{1}{1+y\sqrt{1-p}}\\
            \frac{(2-p)y + 2\sqrt{1-p}}{2\sqrt{1-p}(1+y\sqrt{1-p})^2}
        \end{pmatrix}
    \end{equation*}
    satisfying 
    \begin{equation*}
        \mathbf{L}_p \,\mathbf{g}_{0,p} = \mathbf{f}_{0,p},\quad \mathbf{L}_p^2 \,\mathbf{g}_{0,p} = \mathbf{0}.
    \end{equation*}
\label{spectrum_prop}
\end{proposition}
\begin{proof}
Fix $\varepsilon\in(0,\tfrac12)$ and set $\omega_2:=\tfrac12-\varepsilon>0$. By Proposition \ref{Maximal dissipative prop}, $\widehat{\mathbf L}_p-\omega_2\mathbf I$ is maximally dissipative, hence
\[
\sigma(\widehat{\mathbf L}_p)\subset\{z\in\mathbb C:\mathfrak{Re} z\le -\omega_2\}.
\]
Since $\widehat{\mathbf L}_p=\mathbf L_p-\widehat{\mathbf P}_p$ with $\widehat{\mathbf P}_p$ compact, Kato’s compact-perturbation theorem \cite[Thm.~5.35, p.~244]{kato2013perturbation} yields
\[
\sigma_{\mathrm{ess}}(\mathbf L_p)=\sigma_{\mathrm{ess}}(\widehat{\mathbf L}_p)\subset\{z\in\mathbb C:\mathfrak{Re} z\le -\omega_2\}.
\]
In particular, $0\notin\sigma_{\mathrm{ess}}(\mathbf L_p)$.

Since $\sigma(\mathbf L_p)=\sigma_{\mathrm{ess}}(\mathbf L_p)\cup\sigma_{\mathrm{disc}}(\mathbf L_p)$, the mode stability of Proposition \ref{mode_stability_generalised_solutions} implies that the only eigenvalues with $\mathfrak{Re} z\ge0$ are $\{0,1\}\subset\sigma_{\mathrm{disc}}(\mathbf L_p)$.

Possible spectrum in the strip $\{-\omega_2<\mathfrak{Re} z<0\}$ is therefore discrete. Discrete eigenvalues can only accumulate at points of the essential spectrum (or at infinity), and $0\notin\sigma_{\mathrm{ess}}(\mathbf L_p)$, so there is no accumulation at $0$. Hence the set
\[
\Sigma:=\sigma(\mathbf L_p)\cap\{-\omega_2<\mathfrak{Re} z<0\}
\]
is finite. Define
\[
\omega_0:=\begin{cases}
\min\big\{\omega_2,\; -\max_{\lambda\in\Sigma}\mathfrak{Re}\lambda\big\}, & \Sigma\neq\varnothing,\\[2mm]
\omega_2, & \Sigma=\varnothing.
\end{cases}
\]
Then $\omega_0\in(0,\tfrac12)$ and
\[
\sigma(\mathbf L_p)\subset \{z\in\mathbb C:\Re z\le -\omega_0\}\cup\{0,1\}.
\]

A direct computation shows $\mathbf f_{0,p},\mathbf f_{1,p}\in\mathcal D(\mathbf L_p)$ are eigenfunctions with eigenvalues $0$ and $1$, respectively, and $\mathbf g_{0,p}\in\mathcal D(\mathbf L_p^2)$ is a generalised eigenfunction satisfying $\mathbf L_p\mathbf g_{0,p}=\mathbf f_{0,p}$ and $\mathbf L_p^2\mathbf g_{0,p}=\mathbf 0$. Finally, as proved in the appendix in Proposition \ref{Eigenequation_as_Heun}, any $H^{k+1}(0,1)$ eigenfunction (with $k\ge4$) is smooth, and the geometric eigenspaces at $0$ and $1$ are one-dimensional; hence they are spanned by $\mathbf f_{0,p}$ and $\mathbf f_{1,p}$, respectively.
\end{proof}
Due to the presence of unstable modes, we now introduce the Riesz projectors of isolated eigenvalues $0$ and $1$ for $\mathbf{L}_p$
\begin{align*}
    \mathbf{P}_{0,p} &\coloneqq \int_{\gamma_0}(z\mathbf{I}-\mathbf{L}_p)^{-1}\,dz\\
    \mathbf{P}_{1,p} &\coloneqq \int_{\gamma_1}(z\mathbf{I}-\mathbf{L}_p)^{-1}\,dz
\end{align*}
where we take $\gamma_j:[0,1]\to \mathbb{C}$ as 
\begin{equation*}
    \gamma_0(s) = \frac{\omega_0}{2} e^{2\pi i s}, \quad \gamma_1(s) = 1 + \frac{1}{2}e^{2\pi i s}
\end{equation*}
where the former radius is sufficiently small so that $\gamma_0$ does not intersect $\sigma(\widehat{\mathbf{L}}_p)$.

Recall that the algebraic multiplicity is the rank of the projection $\mathbf{P}_{i,p}$, i.e the dimension of the range of the $\mathbf{P}_{i,p}$. The next result shows that the algebraic multiplicity of the eigenvalue $0$ exceeds the geometric multiplicity.
\begin{lemma}
    The projections $\mathbf{P}_{1,p}$ and $\mathbf{P}_{0,p}$ have rank 1 and 2 respectively. 
\label{rank_multiplicity_lemma}
\end{lemma}
\begin{proof}
The argument is identical to that of \cite{ghoul2025blow} 
and we sketch only the main points for completeness. 
Since $\mathbf{L}_p=\widehat{\mathbf{L}}_p+\mathbf{P}_p$ with $\mathbf{P}_p$ finite rank,
the eigenvalues $0$ and $1$ are isolated of finite algebraic multiplicity 
(\cite[Thm 5.28]{kato2013perturbation}). 
The Riesz projections $\mathbf{P}_{i,p}$ depend continuously on $p\in(0,1)$, 
so their ranks are constant (\cite[Lemma 4.10 ]{kato2013perturbation}). 
It suffices to compute the ranks for a fixed value, say $p=\tfrac34$.

As in \cite{ghoul2025blow}, we have the Jordan chain
$\mathbf{L}_p\mathbf{f}_{0,p}=0$, 
$\mathbf{L}_p\mathbf{g}_{0,p}=\mathbf{f}_{0,p}$, 
and Lemma~\ref{no_more_generalised_eigfunctions} in the Appendix 
rules out any further generalised eigenfunction in $\mathcal{H}^4$.
Hence $\operatorname{rank}\mathbf{P}_{0,p}=2$.  
For $\lambda=1$, Lemma~\ref{rank_1 of 0 eigenvalue} 
shows that $(\mathbf{L}_p-\mathbf{I})\mathbf{v}=\mathbf{f}_{1,p}$ has no solution 
in $\mathcal{D}(\mathbf{L}_p)$, so $\lambda=1$ is simple and 
$\operatorname{rank}\mathbf{P}_{1,p}=1$.  
\end{proof}

     \subsection{A Resolvent Estimate}
For $0<\omega_1<\omega_0$, define the rectangle 
\begin{equation*}
    \Omega_{m,n} \coloneqq \{z\in \mathbb{C}: \mathfrak{Re}z\in [-\omega_1,m], \mathfrak{Im}z \in [-n,n]\}
\end{equation*}
and 
\begin{equation*}
    \Omega_{m,n}' \coloneqq \{z\in \mathbb{C}: \mathfrak{Re}z \geq -\omega_1 \} \backslash \Omega_{m,n}
\end{equation*}
\begin{proposition}
    Take $p_0 \in (0,1)$, and $k\geq 4$, and $\omega_1\in (0,\omega_0)$ and $\varepsilon \in \left(0,\min\{\frac{p_0}{2},\frac{1-p_0}{2}\}\right)$, then there exist $m,n>0$ such that $\Omega_{m,n}'\subset \rho(\mathbf{L}_p)$ and 
    \begin{equation*}
        \sup_{\lambda \in \Omega_{m,n}'}\|(\lambda - \mathbf{L}_p)^{-1}\|_{\mathcal{L}(\mathcal{H}^k)} \leq C 
    \end{equation*}
    for all $p\in \overline{B}_{\varepsilon}(p_0)$, and $C>0$ is a constant independent of $p$. 
\label{resolvent_estimate}
\end{proposition}
\begin{proof}
    The proof follows the same argument to \cite{ghoul2025blow}.
\end{proof}
\section{The Linearised Evolution}
\begin{proposition}
    Take $p_0\in (0,1)$, $k\geq 4$, $\omega_1\in (0,\omega_0)$, $\varepsilon \in \left(0,\min\{\frac{p_0}{2},\frac{1-p_0}{2}\}\right)$ and $p\in \overline{B}_{\varepsilon}(p_0)$, then the following properties hold
    \begin{equation*}
        [\mathbf{S}_p(\tau),\mathbf{P}_{1,p}] = [\mathbf{S}_p(\tau),\mathbf{P}_{0,p}] = \mathbf{0}
    \end{equation*}
    such that 
    \begin{align*}
        \mathbf{S}_p(\tau)\mathbf{P}_{1,p} &= e^\tau \mathbf{P}_{1,p}   \\
        \mathbf{S}_p(\tau)\mathbf{P}_{0,p} &= \mathbf{P}_{0,p} + \tau \mathbf{L}_p\mathbf{P}_{0,p} \\
        \left \|\mathbf{S}_p(\tau)\widetilde{\mathbf{P}}_p \mathbf{q}\right\|_{\mathcal{H}^k} &\leq Me^{-\omega_1\tau}\left \|\widetilde{\mathbf{P}}_p \mathbf{q}\right\|_{\mathcal{H}^k}
    \end{align*}
    for all $\tau \geq 0$, where $\widetilde{\mathbf{P}}_p \coloneqq \mathbf{I} - \mathbf{P}_{1,p} - \mathbf{P}_{0,p}$, and $M = M(p_0,\omega_1)$ is a constant depending only on $p_0,\omega_1$. Moreover, we have 
    \begin{align*}
        \text{ran}(\mathbf{P}_{1,p}) &= \text{span}(\mathbf{f}_{1,p}) \\
        \text{ran}(\mathbf{P}_{0,p}) &= \text{span}(\mathbf{f}_{0,p},\mathbf{g}_{0,p})
    \end{align*}
    and 
    \begin{equation*}
        \mathbf{P}_{0,p}\mathbf{P}_{1,p} = \mathbf{P}_{1,p}\mathbf{P}_{0,p} = \mathbf{0}
    \end{equation*}
\label{linearised_evolution}
\end{proposition}
\begin{proof}
    Any $C_0$-semigroup commutes with its generator, and as a consequence of the Hille Yosida Theorem, it commutes with the resolvent of its generator. Therefore $\mathbf{S}_p(\tau)$ commutes with $\mathbf{P}_{\lambda,p}$ for $\lambda \in \{0,1\}$. This proves the first claim. \\

    \noindent For $\lambda=1$ and $\lambda=0$ respectively, by Lemma \ref{rank_multiplicity_lemma} the dimension of ran$(\mathbf{P}_{\lambda,p})$ is 1 and 2 and so by Proposition \ref{spectrum_prop} the spaces must be a span of $\{\mathbf{f}_{1,p}\}$ and $\{\mathbf{f}_{0,p},\mathbf{g}_{0,p}\}$. \\

    \noindent For the next identities, by definition of the semigroup, for any $\mathbf{q}\in \mathcal{H}^k$ we have
    \begin{equation*}
        \partial_{\tau}\mathbf{S}_p(\tau)\mathbf{q} = \mathbf{L}_p \mathbf{S}_p(\tau)\mathbf{q}
    \end{equation*}
    and by commutation and the eigenvalue property we have in particular 
    \begin{equation*}
        \partial_{\tau}\mathbf{S}_p(\tau)\mathbf{P}_{1,p}\mathbf{q} = \mathbf{L}_p \mathbf{S}_p(\tau)\mathbf{P}_{1,p}\mathbf{q} = \mathbf{S}_p(\tau)\mathbf{L}_p\mathbf{P}_{1,p}\mathbf{q} =  \mathbf{S}_p(\tau)\mathbf{P}_{1,p}\mathbf{q}.
    \end{equation*}
    Together with the initial condition $S_p(\tau)\mathbf{P}_{1,p}\mathbf{q}\rvert_{\tau=0} = \mathbf{P}_{1,p}\mathbf{q}\rvert_{\tau=0}$, the above implies 
    \begin{equation*}
        \mathbf{S}_p(\tau)\mathbf{P}_{1,p} = e^{\tau}\mathbf{P}_{1,p}.
    \end{equation*}
    For the next identity, we first have 
    \begin{align*}
        \partial_{\tau}\mathbf{S}_p(\tau)\mathbf{P}_{0,p}\mathbf{q} &= \mathbf{L}_{p}\mathbf{S}_p(\tau)\mathbf{P}_{0,p}\mathbf{q} = \mathbf{S}_p(\tau)\mathbf{L}_p \mathbf{P}_{0,p}\mathbf{q} \\
        \mathbf{S}_p(\tau)\mathbf{P}_{0,p}\mathbf{q}\rvert_{\tau=0} &= \mathbf{P}_{0,p}\mathbf{q}
    \end{align*}
    Differentiating again,
    \begin{equation*}
        \partial_{\tau \tau} \mathbf{S}_p(\tau)\mathbf{P}_{0,p}\mathbf{q} = \partial_{\tau}\mathbf{S}_p(\tau)\mathbf{L}_p \mathbf{P}_{0,p}\mathbf{q} = \mathbf{L}_p \mathbf{S}_p(\tau)\mathbf{L}_p \mathbf{P}_{0,p}\mathbf{q} = \mathbf{S}_p(\tau)\mathbf{L}_p^2 \mathbf{P}_{0,p}\mathbf{q} = 0
    \end{equation*}
    Together with the other initial condition $\partial_{\tau}\mathbf{S}_p(\tau) \mathbf{P}_{0,p}\mathbf{q}\rvert_{\tau=0} = \mathbf{L}_p \mathbf{P}_{0,p}\mathbf{q}\rvert_{\tau=0}$, the above implies 
    \begin{equation*}
        \mathbf{S}_p(\tau)\mathbf{P}_{0,p} = \mathbf{P}_{0,p} + \tau\mathbf{L}_p \mathbf{P}_{0,p}
    \end{equation*}
    We now prove the $p$-uniform growth estimate of $\mathbf{S}_p(
    \tau)\rvert_{\widetilde{\mathbf{P}}_p \mathcal{H}^k}$. By Theorem 6.17 \cite{kato2013perturbation} we have the decomposition of the spectrum of the reduced operator $\sigma(\mathbf{L}_p\rvert_{\widetilde{\mathbf{P}}_p \mathcal{H}^k}) \subset \{z\in \mathbb{C}:\mathfrak{Re}z\leq -\omega_0\}$ and the resolvent of this reduced operator is now analytic in the box $\Omega_{m,n}$ containing unstable modes, for all $m,n>0$. Now, using Proposition \ref{resolvent_estimate}, the resolvent of the reduced operator satisfies 
    \begin{equation*}
        \left \|(\lambda - \mathbf{L}_p\rvert_{\widetilde{\mathbf{P}}_p \mathcal{H}^k})^{-1}\right\|_{\mathcal{L}(\mathcal{H}^k)}\leq C
    \end{equation*}
    for all $\lambda \in \{z\in \mathbb{C}:\mathfrak{Re}z\geq -\omega_1\}$ and $p\in \overline{B}_{\varepsilon}(p_0)$ where $C$ is independent of $p$. Finally, by the Gearhart-Prüss-Greiner Theorem (\cite{engel2000one}, Theorem 1.11, p.302), this is equivalent to uniform exponential stability, i.e the growth estimate with growth bound $-\omega_1$. 
\end{proof}
\section{Nonlinear Stability}
\noindent We now write the full nonlinear system for the perturbation $\mathbf{q}$
\begin{equation}
    \begin{cases}
        \partial_{\tau}\mathbf{q} = \mathbf{L}_p \mathbf{q} + \mathbf{N}(\mathbf{q}) \\
        \mathbf{q}(\tau=0,y)=\mathbf{q}_0(y)
    \end{cases}
\label{abstract_cauchy_problem}
\end{equation}
where 
\begin{equation*}
        \mathbf{N}(\mathbf{q}) = \begin{pmatrix}
            0 \\
            q_2^2 
        \end{pmatrix}.
\end{equation*}
By Duhamel's principle, we have 
\begin{equation*}
    \mathbf{q}(\tau,\cdot) = \mathbf{S}_p(\tau)\mathbf{q}_0(\cdot) + \int_{0}^{\tau}\mathbf{S}_p(\tau-\tau')\mathbf{N}(\mathbf{q}(\tau'))\,d\tau'. 
\end{equation*}
\subsection{Stabilised Nonlinear Evolution}
\begin{definition}
    For $k\geq 4$, and $\omega_1 \in (0,\omega_0)$ arbitrary but fixed, we define the Banach space  $\left(\mathcal{X}^k(-1,1),\|\cdot\|_{\mathcal{X}^k(-1,1)}\right)$ by 
    \begin{align*}
        \mathcal{X}^k(-1,1) &\coloneqq \{\mathbf{q}\in C\left([0,\infty);\mathcal{H}^k(-1,1) \right):\|\mathbf{q}(\tau)\|_{\mathcal{H}^k}\lesssim e^{-\omega_1 \tau} \,\, \text{for all} \, \,\tau\geq 0\} \\
        \|\mathbf{q}\|_{\mathcal{X}^k(-1,1)} &\coloneqq \sup_{\tau \geq 0}\left(e^{\omega_1 \tau}\|\mathbf{q}\|_{\mathcal{H}^k} \right)
    \end{align*}
    We also define the closed balls of size $\varepsilon >0$
    \begin{align*}
        \mathcal{H}_{\varepsilon}^k(-1,1) &\coloneqq \{\mathbf{f}\in \mathcal{H}^k(-1,1):\|\mathbf{f}\|_{\mathcal{H}^k(-1,1)}\leq \varepsilon\} \\
        \mathcal{X}_{\varepsilon}^k(-1,1) &\coloneqq \{\mathbf{q}\in \mathcal{X}^k(-1,1): \|\mathbf{q}\|_{\mathcal{X}^k(-1,1)}\leq \varepsilon\}
    \end{align*}
\end{definition}
\noindent By formally applying the projections $\mathbf{P}_{\lambda,p}$ to the Duhamel formula, and using the properties of Proposition \ref{linearised_evolution} and writing $\mathbf{P}_p \coloneqq \mathbf{P}_{1,p} + \mathbf{P}_{0,p}$ we obtain the correction 
\begin{equation*}
    \mathbf{C}_p(\mathbf{f},\mathbf{q}) = \mathbf{P}_p\mathbf{f} + \mathbf{P}_{0,p}\int_{0}^{\infty}\mathbf{N}(\mathbf{q}(\tau'))\,d\tau' + \mathbf{L}_p\mathbf{P}_{0,p}\int_{0}^{\infty}(-\tau')\mathbf{N}(\mathbf{q}(\tau'))\,d\tau' + \mathbf{P}_{1,p}\int_{0}^{\infty}e^{-\tau'}\mathbf{N}(\mathbf{q}(\tau'))\,d\tau' 
\end{equation*}
We note that $\mathbf{C}(\mathbf{f},\mathbf{q}) \in \text{ran}(\mathbf{P}_p) \subset \mathcal{H}^k$. We then anticipate that by subtracting the initial data by this term, we will be able to close a fixed point argument. Before this we state an estimate on the nonlinearity $\mathbf{N}$.
\begin{lemma}
    for $k\geq 1$ we have
    \begin{align*}
        \|\mathbf{N}(\mathbf{q})\|_{\mathcal{H}^k} &\leq \|\mathbf{q}\|_{\mathcal{H}^k}^2 \\
        \|\mathbf{N}(\mathbf{q})-\mathbf{N}(\widetilde{\mathbf{q}})\|_{\mathcal{H}^k} &\lesssim \|\mathbf{q}-\widetilde{\mathbf{q}}\|_{\mathcal{H}^k}(\|\mathbf{q}\|_{\mathcal{H}^k}+\|\widetilde{\mathbf{q}}\|_{\mathcal{H}^k})
    \end{align*}
\label{nonlinear estimates}
\end{lemma}
\noindent We now prove the existence of a small perturbation along the stable subspace.
\begin{proposition}
    Take $p_0 \in (0,1)$, $k\geq 4$, and $\omega_1 \in (0,\omega_0)$. There are constants $0<\varepsilon_0<\min\{\frac{p_0}{2},\frac{1-p_0}{2}\}$ and $C_0 = C_0(p,k,\omega_1)>1$ such that for all $p\in \overline{B}_{\varepsilon_0}(p_0)$ and all $\mathbf{f}\in \mathcal{H}^k(-1,1)$ with $\|\mathbf{f}\|_{\mathcal{H}^k}\leq \frac{\varepsilon_0}{C_0}$, there is a unique global solution $\mathbf{q}_p \in \mathcal{X}^k(-1,1)$ satisfying $\|\mathbf{q}_p\|_{\mathcal{X}^k}\leq \varepsilon_0$ and 
    \begin{equation*}
        \mathbf{q}_p(\tau) = \mathbf{S}_p(\tau)(\mathbf{f}-\mathbf{C}_p(\mathbf{f},\mathbf{q})) + \int_{0}^{\tau}\mathbf{S}_p(\tau-\tau')\mathbf{N}(\mathbf{q}(\tau'))\,d\tau' 
    \end{equation*}
    for all $\tau \geq 0$. Moreover, the data-to-solution map 
    \begin{equation*}
        \mathcal{H}_{\frac{\varepsilon_0}{C_0}}^k(-1,1) \to \mathcal{X}^k(-1,1), \quad \mathbf{f}\mapsto \mathbf{q}_p 
    \end{equation*}
    is Lipschitz continuous. 
\label{Correction_duhamel_prop}
\end{proposition}
\begin{proof}
    We take $p_0 \in(0,1)$, $k\geq 4$ and $\varepsilon_0 \in (0,\min\{\frac{p_0}{2},\frac{1-p_0}{2}\})$ and $\omega_1 \in (0,\omega_0)$ arbitrary but fixed so that for all $p\in \overline{B}_{\varepsilon_0}(p_0)$ the semigroup $\mathbf{S}_p(\tau)$ is well-defined and satisfies Proposition \ref{linearised_evolution}. We denote the solution operator $\mathbf{K}_p(\mathbf{f},\mathbf{q})$, depending on initial data $\mathbf{f}$ and solution $\mathbf{q}$, as
    \begin{equation*}
        \mathbf{K}_p(\mathbf{f},\mathbf{q})(\tau) = \mathbf{S}_p(\tau)(\mathbf{f}-\mathbf{C}_p(\mathbf{f},\mathbf{q})) + \int_{0}^{\tau}\mathbf{S}_p(\tau-\tau')\mathbf{N}(\mathbf{q}(\tau'))\,d\tau'.
    \end{equation*}
    Using the expression for the correction $\mathbf{C}(\mathbf{f},\mathbf{q})$, we rewrite this, recalling that $\widetilde{\mathbf{P}}_p \coloneqq \mathbf{I} -\mathbf{P}_p$, as 
    \begin{align}
\begin{split}
        \mathbf{K}_p(\mathbf{f},\mathbf{q})(\tau) &= \mathbf{S}_p(\tau)\widetilde{\mathbf{P}}_p \mathbf{f} 
        + \int_{0}^{\tau}\mathbf{S}_p(\tau-\tau')\widetilde{\mathbf{P}}_p\mathbf{N}(\mathbf{q}(\tau'))\,d\tau' \\
        &\quad - \mathbf{P}_{0,p}\int_{\tau}^{\infty}\mathbf{N}(\mathbf{q}(\tau'))\,d\tau' 
        - \mathbf{L}_p\mathbf{P}_{0,p}\int_{\tau}^{\infty}(\tau - \tau')\mathbf{N}(\mathbf{q}(\tau'))\,d\tau' \\
        &\quad - \mathbf{P}_{1,p}\int_{\tau}^\infty e^{\tau-\tau'}\mathbf{N}(\mathbf{q}(\tau'))\,d\tau'
\end{split}
\label{representation_for_K}
\end{align}
    We now show the first step of the fixed point argument. Namely, that for $\varepsilon_0>0$, $\mathbf{f} \in \mathcal{H}^k$ sufficiently small, that  $\mathbf{K}_p(\mathbf{f},\mathbf{q})$ maps from the ball $\mathcal{X}_{\varepsilon_0}^k$ to itself. Indeed, by Proposition \ref{linearised_evolution}, and the nonlinear estimates in Lemma \ref{nonlinear estimates} we have 
    \begin{align*}
        \|\mathbf{K}_p(\mathbf{f},\mathbf{q})\|_{\mathcal{H}^k}(\tau) &\lesssim e^{-\omega_1 \tau}\|\mathbf{f}\|_{\mathcal{H}^k} + \int_{0}^\tau e^{-\omega_1(\tau-\tau')}\|\mathbf{q}(\tau')\|_{\mathcal{H}^k}^2\,d\tau' + \int_{\tau}^{\infty}\|\mathbf{q}(\tau')\|_{\mathcal{H}^k}^2\,d\tau' \\
        &+ \int_{\tau}^{\infty}(\tau'-\tau)\|\mathbf{q}(\tau')\|_{\mathcal{H}^k}^2\,d\tau' + \int_{\tau}^{\infty}e^{\tau-\tau'}\|\mathbf{q}(\tau')\|_{\mathcal{H}^k}^2\,d\tau' \\
        &\lesssim e^{-\omega_1\tau}\frac{\varepsilon_0}{C_0} +\varepsilon_0^2 \int_{0}^{\tau}e^{-\omega_1(\tau-\tau')} e^{-2\omega_1 \tau'}\,d\tau' + \varepsilon_0^2 \int_{\tau}^{\infty} e^{-2\omega_1\tau'}\,d\tau' \\
        &+ \varepsilon_0^2 \int_{\tau}^{\infty}(\tau'-\tau)e^{-2\omega_1 \tau'}\,d\tau' + \varepsilon_0^2 \int_{\tau}^{\infty}e^{\tau-\tau'}e^{-2\omega_1 \tau'}\,d\tau' \\
        &\lesssim e^{-\omega_1 \tau}\left(\frac{\varepsilon_0}{C}+\varepsilon_0^2\right) \\
        &\leq e^{-\omega_1 \tau}\varepsilon_0
    \end{align*}
    for all $\tau \geq 0$, $\varepsilon_0 \in (0,\min\{\frac{p_0}{2},\frac{1-p_0}{2}\})$ sufficiently small, $p\in \overline{B}_{\varepsilon_0}(p_0)$, $\mathbf{f} \in \mathcal{H}^k$ satisfying $\|\mathbf{f}\|_{\mathcal{H}^k}\leq \frac{\varepsilon_0}{C_0}$ and $\mathbf{q}\in \mathcal{X}_{\varepsilon_0}^k$ and any $C_0>1$. This shows that for such $\mathbf{f}\in \mathcal{H}_{\frac{\varepsilon_0}{C_0}}^k$, $\mathbf{K}_p(\mathbf{f},\cdot): \mathcal{X}_{\varepsilon_0}^k\to \mathcal{X}_{\varepsilon_0}^k$. \\

    \noindent We now show that $\mathbf{K}_p(\mathbf{f},\cdot)$ is a contraction. Using \eqref{representation_for_K} and Lemma \ref{nonlinear estimates} we have
\begin{align*}
&\|\mathbf{K}_p(\mathbf{f},\mathbf{q})(\tau) - \mathbf{K}_p(\mathbf{f},\widetilde{\mathbf{q}})(\tau)\|_{\mathcal{H}^k} \\[0.5em]
&\quad \lesssim \int_{0}^{\tau} e^{-\omega_1(\tau-\tau')} \|\mathbf{N}(\mathbf{q})(\tau') - \mathbf{N}(\widetilde{\mathbf{q}})(\tau')\|_{\mathcal{H}^k} \, d\tau' 
+ \int_{\tau}^{\infty} \|\mathbf{N}(\mathbf{q})(\tau') - \mathbf{N}(\widetilde{\mathbf{q}})(\tau')\|_{\mathcal{H}^k} \, d\tau' \\ 
&\quad + \int_{\tau}^{\infty} (\tau'-\tau)\, \|\mathbf{N}(\mathbf{q})(\tau') - \mathbf{N}(\widetilde{\mathbf{q}})(\tau')\|_{\mathcal{H}^k} \, d\tau' 
+ \int_{\tau}^{\infty} e^{\tau-\tau'} \|\mathbf{N}(\mathbf{q})(\tau') - \mathbf{N}(\widetilde{\mathbf{q}})(\tau')\|_{\mathcal{H}^k} \, d\tau' \\[0.5em]
&\quad \lesssim \epsilon_0 \Big[ \int_{0}^{\tau} e^{-\omega_1(\tau-\tau')} e^{-2\omega_1 \tau'} \, d\tau' 
+ \int_{\tau}^{\infty} e^{-2\omega_1 \tau'} \, d\tau'
+ \int_{\tau}^{\infty} (\tau'-\tau) e^{-2\omega_1 \tau'} \, d\tau'
+ \int_{\tau}^{\infty} e^{\tau-\tau'} e^{-2\omega_1 \tau'} \, d\tau' \Big] \|\mathbf{q}-\widetilde{\mathbf{q}}\|_{\mathcal{X}^k} \\[0.5em]
&\quad \lesssim \epsilon_0 e^{-\omega_1 \tau} \|\mathbf{q}-\widetilde{\mathbf{q}}\|_{\mathcal{X}^k}.
\end{align*}
for all $\mathbf{q},\widetilde{\mathbf{q}}\in \mathcal{X}_{\varepsilon_0}^k$, $\tau\geq 0$, $p \in \overline{B}_{\varepsilon_0}(p_0)$ and $\mathbf{f}\in \mathcal{H}^k$. We can now choose $\varepsilon_0$ small enough so that 
\begin{equation*}
    \|\mathbf{K}_p(\mathbf{f},\mathbf{q}) - \mathbf{K}_p(\mathbf{f},\widetilde{\mathbf{q}})\|_{\mathcal{X}^k}\leq \frac{1}{2}\|\mathbf{q}-\widetilde{\mathbf{q}}\|_{\mathcal{X}^k}
\end{equation*}

\noindent By Banach's fixed point theorem, for $\mathbf{f}\in \mathcal{H}_{\frac{\varepsilon_0}{C_0}}^k$, there exists a unique fixed point $\mathbf{q}_p\in \mathcal{X}_{\varepsilon_0}^k$ such that $\mathbf{K}_p(\mathbf{f},\mathbf{q}_p)=\mathbf{q}_p$. Finally for Lipschitz continuity, we take $\mathbf{f},\widetilde{\mathbf{f}}\in \mathcal{H}_{\frac{\varepsilon_0}{C_0}}^k$ and the associated $\mathbf{K}_p$-fixed points $\mathbf{q}_p,\widetilde{\mathbf{q}}_p \in \mathcal{X}_{\varepsilon_0}^k$, then firstly we write 
\begin{align*}
    \mathbf{q}_p(\tau) -\widetilde{\mathbf{q}}_p(\tau) &= \mathbf{K}_p(\mathbf{f},\mathbf{q}_p)(\tau) - \mathbf{K}_p(\widetilde{\mathbf{f}},{\widetilde{\mathbf{q}}_p})(\tau) \\
    &= \mathbf{K}_p(\mathbf{f},\mathbf{q}_p)(\tau) - \mathbf{K}_p(\mathbf{f},{\widetilde{\mathbf{q}}_p})(\tau) + \mathbf{S}_p(\tau)\widetilde{\mathbf{P}}_p(\mathbf{f}-\widetilde{\mathbf{f}}).
\end{align*}
Using the contraction estimate we have 
\begin{align*}
    \frac{1}{2}\|\mathbf{q}_p -\widetilde{\mathbf{q}}_p\|_{\mathcal{X}} \leq \|\mathbf{S}_p(\tau)\widetilde{\mathbf{P}}_p(\mathbf{f}-\widetilde{\mathbf{f}})\|_{\mathcal{X}^k} \lesssim \|\mathbf{f}-\widetilde{\mathbf{f}}\|_{\mathcal{X}^k}
\end{align*}
which proves Lipschitz continuous dependence on the initial data. 
\end{proof}
\subsection{Stable flow near the blow up solution}
Recall that we actually want to solve the Duhamel equation without the correction term $\mathbf{C}(\mathbf{f},\mathbf{q})$. To do this we take modulation parameters $(p_0,T_0,\kappa_0)$ that we will later choose so that the correction term vanishes. 
\begin{definition}
    Take $p,p_0 \in(0,1)$, $T,T_0 >0$, $\kappa,\kappa_0 \in \mathbb{R}$ and $k\geq 4$, with $T\sqrt{1-p_0}<T$. We define the $p,T,\kappa$-dependent initial data operator 
    \begin{equation*}
        \mathbf{U}_{p,T,\kappa} : \mathcal{H}^k(-1,1) \to \mathcal{H}^k(-1,1), \quad \mathbf{f}\mapsto \mathbf{f}^{T} + \mathbf{f}_0^T - \mathbf{f}_{p,\kappa}
    \end{equation*}
    where 
    \begin{equation*}
        \mathbf{f} = \begin{pmatrix}
        f_1(y) \\
        f_2(y)
    \end{pmatrix}
    \end{equation*}
    and 
    \begin{equation*}
        \mathbf{f}^T(y) = \begin{pmatrix}
            f_1(Ty) \\
            Tf_2(Ty)
        \end{pmatrix},\quad \mathbf{f}_0^T(y) = \begin{pmatrix}
            \widetilde{U}_{p_0,1,\kappa_0}(\frac{T}{T_0}y) \\
            \frac{T}{T_0}p_0 + \left(\frac{T}{T_0}\right)^2y \partial_y \widetilde{U}_{p_0,1,\kappa_0}(\frac{T}{T_0}y)
        \end{pmatrix},\quad
        \mathbf{f}_{p,\kappa}(y) =
        \begin{pmatrix}
             \widetilde{U}_{p,1,\kappa}(y) \\
             p+y\partial_y \widetilde{U}_{p,1,\kappa}(y)
        \end{pmatrix}
    \end{equation*}
\end{definition}
We remark that this operator prepares the initial data for the perturbation for the abstract Cauchy problem and splits the initial data into $\mathbf{f}^T$ and $\mathbf{f}_0^T$, the initial data for a blow up solution with parameters $p_0,T_0,\kappa_0$. \\
\begin{lemma}
    Let $p,p_0\in (0,1)$, $T,T_0>0$, $\kappa,\kappa_0\in \mathbb{R}$, $k\geq 4$, and $T<T_0\sqrt{1-p_0}$. Then 
    \begin{equation*}
        \mathbf{U}_{p,T,\kappa}(\mathbf{f}) = \mathbf{f}^T + \left[(\kappa-\kappa_0) -p\left(\frac{T}{T_0}-1\right) +\frac{p}{2(1-p)}(p_0-p)\right]\mathbf{f}_{0,p} - \frac{\left(\frac{T}{T_0}-1\right)}{\sqrt{1-p}}\mathbf{f}_{1,p} + (p_0-p)\mathbf{g}_{0,p} + \mathbf{r}\left(p,\frac{T}{T_0}\right)
    \end{equation*}
    for any $\mathbf{f}\in \mathcal{H}^k$ and $\mathbf{f}_{0,p},\mathbf{f}_{1,p},\mathbf{g}_{0,p}$ are introduced in Proposition \ref{spectrum_prop} and 
    \begin{equation*}
        \|\mathbf{r}\left(p,\frac{T}{T_0}\right)\|_{\mathcal{H}^k} \lesssim \left(\frac{T}{T_0}-1\right)^2 + \left(p-p_0\right)^2
    \end{equation*}
    for $p\in \overline{B}_{\varepsilon_1}(p_0)$ and $T\in \overline{B}_{T_0\varepsilon}(T_0)$ where $\varepsilon_1 = \min\left\{\frac{p_0}{2},\frac{1-p_0}{2},\frac{\delta}{2}\right\}$, with $\delta \coloneqq \sqrt{1-p_0}-\frac{T}{T_0}$. 
\label{initial_data_lemma}
\end{lemma}
\begin{proof}
    Recall $\widetilde{U}_{p,1,\kappa} = -p\log(1+y\sqrt{1-p})+\kappa$. Define the distance $\delta \coloneqq \sqrt{1-p_0} - \frac{T}{T_0}$, then let $\varepsilon_1 = \min\{\frac{p_0}{2},\frac{1-p_0}{2},\frac{\delta}{2}\}$ and consider $p\in \overline{B}_{\varepsilon_1}(p_0)$ and $T\in \overline{B}_{T_0\varepsilon}(T_0)$. For fixed $y\in [-1,1]$, a Taylor expansion around $(p_0,\kappa_0,T_0)$ in each component gives 
\begin{equation*}
    \mathbf{f}_{0}^T - \mathbf{f}_{p,\kappa} = \left[(\kappa-\kappa_0) -p_0\left(\frac{T}{T_0}-1\right) +\frac{p_0}{2(1-p_0)}(p_0-p)\right]\mathbf{f}_{0,p_0} - \frac{\left(\frac{T}{T_0}-1\right)}{\sqrt{1-p_0}}\mathbf{f}_{1,p_0} + (p_0-p)\mathbf{g}_{0,p_0} + \widetilde{\mathbf{r}}\left(p,\frac{T}{T_0}\right)
\end{equation*}
where 
\begin{align*}
    \left[\widetilde{\mathbf{r}}\left(p,\frac{T}{T_0}\right)\right]_i &= \left(\frac{T}{T_0}-1\right)^2\int_{0}^{1}T_0^2\left[\partial_{T'T'}[\mathbf{f}_{0}^{T'}]_i\rvert_{T'=T_0+z(T-T_0)}\right](1-z)\,dz \\
    &- (p-p_0)^2\int_{0}^{1}\left[\partial_{p'p'}[\mathbf{f}_{p',\kappa}]_i\rvert_{p'=p+z(p-p_0)}\right](1-z)\,dz 
\end{align*}
for $i=1,2$.
Thus 
\begin{equation*}
    \mathbf{f}_{0}^T - \mathbf{f}_{p,\kappa} = \left[(\kappa-\kappa_0) -p\left(\frac{T}{T_0}-1\right) +\frac{p}{2(1-p)}(p_0-p)\right]\mathbf{f}_{0,p} - \frac{\left(\frac{T}{T_0}-1\right)}{\sqrt{1-p}}\mathbf{f}_{1,p} + (p_0-p)\mathbf{g}_{0,p} + \mathbf{r}\left(p,\frac{T}{T_0}\right)
\end{equation*}
where
\begin{align*}
    \mathbf{r}\left(p,\frac{T}{T_0}\right) &= \widetilde{\mathbf{r}}(p,T) + \left[(\kappa-\kappa_0)-\left(\frac{T}{T_0}-1\right)(p-p_0) + \left(\frac{p_0}{2(1-p_0)}-\frac{p}{2(1-p)}\right)(p_0-p)\right](\mathbf{f}_{0,p_0}-\mathbf{f}_{0,p}) \\
    &+ \left(\frac{1}{\sqrt{1-p}}-\frac{1}{\sqrt{1-p_0}}\right)\left(\frac{T}{T_0}-1\right)(\mathbf{f}_{1,p}-\mathbf{f}_{1,p_0}) + (p-p_0)(\mathbf{g}_{0,p_0}-\mathbf{g}_{0,p}) \\
    &= \left(\frac{1}{\sqrt{1-p}}-\frac{1}{\sqrt{1-p_0}}\right)\left(\frac{T}{T_0}-1\right)(\mathbf{f}_{1,p}-\mathbf{f}_{1,p_0}) + (p-p_0)(\mathbf{g}_{0,p_0}-\mathbf{g}_{0,p})
\end{align*}
Note that the remainder $\widetilde{\mathbf{r}}(p,T)$ contains the terms $\partial_{TT}\mathbf{f}_{0}^T$ and $\partial_{pp}\mathbf{f}_{p,\kappa}$ which are smooth by construction, and are uniformly bounded in $[-1,1]$ for all $p\in \overline{B}_{\varepsilon_1}(p_0)$ and $T\in \overline{B}_{T_0\varepsilon}(T_0)$. Applying Youngs inequality to the above yields the bound
\begin{equation*}
    \|\mathbf{r}\left(p,\frac{T}{T_0}\right)\|_{\mathcal{H}^k} \lesssim \left(\frac{T}{T_0}-1\right)^2 + \left(p-p_0\right)^2.
\end{equation*}
\end{proof}
\begin{proposition}
    Let $p_0 \in (0,1)$, $T_0>0$, $\kappa_0 \in \mathbb{R}$, $k\geq 4$, and $0<\omega_1<\omega_0$. There are constants $\varepsilon_2>0$ and $C_2>1$ such that for all real-valued $\mathbf{f}\in \mathcal{H}^k(-1,1)$ with $\|\mathbf{f}\|_{\mathcal{H}^k}\leq \frac{\varepsilon_2}{C_2^2}$, there exists parameters $p^* \in (0,1)$, $T_0>0$, $\kappa^* \in \mathbb{R}$, and a unique $\mathbf{q}_{p^*,\kappa^*,T^*}\in C\left([0,\infty),\mathcal{H}^k(-1,1)\right)$ with $\|\mathbf{q}_{p^*,\kappa^*,T^*}(\tau)\|_{\mathcal{H}^k} \leq \varepsilon_2e^{-\omega_1 \tau}$ and 
    \begin{equation}
        \mathbf{q}_{p^*,\kappa^*,T^*} = \mathbf{S}_{p^*}(\tau)\mathbf{U}_{p^*,\kappa^*,T^*}(\mathbf{f}) + \int_{0}^{\tau}\mathbf{S}_{p^*}(\tau-\tau')\mathbf{N}(\mathbf{q}_{p^*,\kappa^*,T^*}(\tau'))\,d\tau'
    \label{No_correction_duhamel}
    \end{equation}
    with 
    \begin{equation*}
        |p^* -p_0| + |\kappa^* -\kappa| + \left|\frac{T}{T_0}-1\right| \leq \frac{\varepsilon_2}{C_2}
    \end{equation*}
\label{Cramer_Brouer_prop}
\end{proposition}
\begin{proof}
Take the $\varepsilon_0>0$ and $C_0 > 1$ obtained from Proposition \ref{Correction_duhamel_prop} and $\varepsilon_1>0$ from Lemma \ref{initial_data_lemma}. We let $0<\varepsilon' \leq \min\{\varepsilon_0,\varepsilon_1\}$ and a $C'\geq C_0$. We further define $\varepsilon_2$ and $C_2$ via 
\begin{align*}
    \varepsilon_2 \coloneqq \frac{\varepsilon'}{M} \quad C_2 \coloneqq MC'
\end{align*}
for some $M\geq 1$ which we will choose later. 
First we bound $\mathbf{U}_{p,T,\kappa}(\mathbf{f})$ for $\mathbf{f}\in \mathcal{H}^k$ with $\|\mathbf{f}\|_{\mathcal{H}^k}\leq \frac{\varepsilon_2}{C_2^2}$, $p\in \overline{B}_{\frac{\varepsilon_2}{C_2}}(p_0)$, $T\in \overline{B}_{\frac{T_0\varepsilon_2}{C_2}}(T_0)$ and $\kappa \in \overline{B}_{\frac{\varepsilon_2}{C_2}}(\kappa_0)$. Indeed, 
\begin{align*}
    \|\mathbf{U}_{p,T,\kappa}(\mathbf{f})\|_{\mathcal{H}^k} &\leq \|\mathbf{f}\|_{\mathcal{H}^k} + \left(|\kappa - \kappa_0| + p\left|\frac{T}{T_0}-1\right| + \frac{p}{2(1-p)}|p-p_0|\right)\|\mathbf{f}_{0,p}\|_{\mathcal{H}^k} + \frac{1}{\sqrt{1-p}}\left|\frac{T}{T_0}-1\right|\|\mathbf{f}_{1,p}\|_{\mathcal{H}^k}\\
    &+ |p-p_0|\|\mathbf{g}_{0,p}\|_{\mathcal{H}^k} + \|\mathbf{r}\left(p,\frac{T}{T_0}\right)\|_{\mathcal{H}^k} \\
    &\leq \frac{\varepsilon_2}{C_2^2} + \left(\frac{\varepsilon_2}{C_2}+ \frac{\varepsilon_2}{C_2} + c_1(p_0)\frac{\varepsilon_2}{C_2}\right) + c_2(p_0)\frac{\varepsilon_2}{C_2} + c_3(p_0)\frac{\varepsilon_2}{C_2} + c_4(p_0) \frac{\varepsilon_2^2}{C_2^2} \\
    &\leq \frac{\varepsilon'}{C'}
\end{align*}
for $M$ sufficiently large. We conclude that it is possible to fix $\varepsilon'>0$ and $C'>1$, so that $\mathbf{U}_{p,T,\kappa}$ satisfies the assumptions of the initial data in Proposition \ref{Correction_duhamel_prop} which implies the existence and uniqueness of a solution $\mathbf{q}_{p,T,\kappa} \in \mathcal{X}^k(-1,1)$ with $\|\mathbf{q}_{p,T,\kappa}\|_{\mathcal{X}^k}\leq \varepsilon_2$ satisfying 
\begin{equation*}
    \mathbf{q}_{p,T,\kappa}(\tau) = \mathbf{S}_p(\tau)\left(\mathbf{U}_{p,T,\kappa}(\mathbf{f}) - \mathbf{C}_p(\mathbf{U}_{p,T,\kappa}(\mathbf{f}),\mathbf{q}_{p,T,\kappa}(\tau)) \right) + \int_{0}^{\tau}\mathbf{S}_p(\tau-\tau')\mathbf{N}(\mathbf{q}_{p,T,\kappa}(\tau'))\,d\tau' 
\end{equation*}
for all $p\in \overline{B}_{\frac{\varepsilon_2}{C_2}}(p_0)$, $T\in \overline{B}_{\frac{T_0\varepsilon_2}{C_2}}(T_0)$ and $\kappa \in \overline{B}_{\frac{\varepsilon_2}{C_2}}(\kappa_0)$. It remains to show the existence of parameters $(p,T,\kappa)$ such that the correction $\mathbf{C}_p(\mathbf{U}_{p,T,\kappa}(\mathbf{f}),\mathbf{q}_{p,T,\kappa})\in \text{ran}(\mathbf{P}_p)$ is equal to $\mathbf{0}$. By Proposition \ref{spectrum_prop}, we have that $\text{ran}(\mathbf{P}_p)\subset \mathcal{H}^k$ is a finite dimensional sub-Hilbert space spanned by the linearly independent functions $\{\mathbf{g}_{0,p},\mathbf{f}_{0,p},,\mathbf{f}_{1,p}\}$. It suffices to show that there are parameters for which the linear functional 
\begin{equation*}
    \ell_{p,T,\kappa}:\text{ran}(\mathbf{P}_p) \to \mathbb{R}, \quad \mathbf{g}\mapsto \left(\mathbf{C}_p(\mathbf{U}_{p,T,\kappa}(\mathbf{f}),\mathbf{q}_{p,T,\kappa}),\mathbf{g} \right)_{\mathcal{H}^k}
\end{equation*}
is zero on a basis of $\text{ran}(\mathbf{P}_p)$. Indeed, this would imply $\mathbf{C}_p(\mathbf{U}_{p,T,\kappa}(\mathbf{f}),\mathbf{q}_{p,T,\kappa}) \perp \text{ran}(\mathbf{P}_p)$, which yields the result. Using the expression for the correction and Lemma \ref{initial_data_lemma}, we have 
\begin{align*}
    \ell_{p,\kappa,T}(\mathbf{g}) &= (\mathbf{P}_p \mathbf{f}^T,\mathbf{g})_{\mathcal{H}^k} + \left[(\kappa-\kappa_0)-p\left(\frac{T}{T_0}-1\right)+\frac{p}{2(1-p)}(p_0-p)\right](\mathbf{f}_{0,p},\mathbf{g})_{\mathcal{H}^k} - \frac{\left(\frac{T}{T_0}-1 \right)}{\sqrt{1-p}}(\mathbf{f}_{1,p},\mathbf{g})_{\mathcal{H}^k} \\
    &+ (p_0-p)(\mathbf{g}_{0,p},\mathbf{g})_{\mathcal{H}^k} + \left(\mathbf{P}_p \mathbf{r}\left(p,\frac{T}{T_0}\right),\mathbf{g}\right)_{\mathcal{H}^k} + \left(\mathbf{P}_{0,p}\int_{0}^{\infty}\mathbf{N}(\mathbf{q}(\tau'))\,d\tau',\mathbf{g} \right)_{\mathcal{H}^k} \\
    &+ \left(\mathbf{L}_p\mathbf{P}_{0,p}\int_{0}^{\infty}(-\tau')\mathbf{N}(\mathbf{q}(\tau'))\,d\tau',\mathbf{g}\right)_{\mathcal{H}^k} + \left(\mathbf{P}_{1,p}\int_{0}^{\infty}e^{-\tau'}\mathbf{N}(\mathbf{q}(\tau'))\,d\tau',\mathbf{g} \right)_{\mathcal{H}^k}.
\end{align*}
We now want to construct a dual basis $\{\mathbf{g}_p^1,\mathbf{g}_p^2,\mathbf{g}_p^3\}$ to the basis $\{\mathbf{g}_{0,p},\mathbf{f}_{0,p},\mathbf{f}_{1,p}\}$. To do this we use the Gram matrix 
\begin{equation*}
    \Gamma(p) \;=\;
    \begin{pmatrix}
        (\mathbf{g}_{0,p},\mathbf{g}_{0,p})_{\mathcal{H}^k} & (\mathbf{g}_{0,p},\mathbf{f}_{0,p})_{\mathcal{H}^k} & (\mathbf{g}_{0,p},\mathbf{f}_{1,p})_{\mathcal{H}^k} \\
        (\mathbf{f}_{0,p},\mathbf{g}_{0,p})_{\mathcal{H}^k} & (\mathbf{f}_{0,p},\mathbf{f}_{0,p})_{\mathcal{H}^k} & (\mathbf{f}_{0,p},\mathbf{f}_{1,p})_{\mathcal{H}^k} \\
        (\mathbf{f}_{1,p},\mathbf{g}_{0,p})_{\mathcal{H}^k} & (\mathbf{f}_{1,p},\mathbf{f}_{0,p})_{\mathcal{H}^k} & (\mathbf{f}_{1,p},\mathbf{f}_{1,p})_{\mathcal{H}^k}
    \end{pmatrix}.
\end{equation*}
with components $\Gamma(p)_{ij}$, with $i,j=1,2,3$. Since the eigenfunctions are linearly independent for $p \in (0,1)$, the Gram matrix is invertible and we denote the elements of the inverse of $\Gamma(p)$ as $\Gamma(p)^{mn}$ with $m,n = 1,2,3$. By defining 
\begin{equation*}
    \mathbf{g}_p^n \coloneqq \Gamma(p)^{n1}\mathbf{g}_{0,p} + \Gamma(p)^{n2}\mathbf{f}_{0,p}+ \Gamma(p)^{n3}\mathbf{f}_{1,p} \quad n=1,2,3
\end{equation*}
we get 
\begin{equation*}
    (\mathbf{g}_{0,p},\mathbf{g}_p^j)_{\mathcal{H}^k} = \delta_1^j, \quad (\mathbf{f}_{0,p},\mathbf{g}_p^j) = \delta_2^j, \quad (\mathbf{f}_{1,p},\mathbf{g}_p^j)_{\mathcal{H}^k} = \delta_3^j 
\end{equation*}
for $j=1,2,3$. Moreover, by Cramer's rule, the components of $\{\mathbf{g}_p^1,\mathbf{g}_p^2,\mathbf{g}_p^3\}\subset \text{ran}(\mathbf{P}_p)$ depend smoothly on $p$ since $\Gamma(p)$ depends smoothly on $p$ and the determinant is non-zero. We now define the continuous map
\begin{align*}
F:\overline{B}_{\frac{\varepsilon_2}{C_2}}(p_0) \times \overline{B}_{\frac{\varepsilon_2}{C_2}}(\kappa_0) \times \overline{B}_{\frac{T_0 \varepsilon_2}{C_2}}(T_0) \to \mathbb{R}^3
\end{align*}
\begin{align*}
(p,\kappa,T) \mapsto \left(p_0 + F_1, \kappa_0 + p\left(\frac{T}{T_0}-1\right) -\frac{p}{2(1-p)}(p_0-p) - F_2,T_0(1+\sqrt{1-p}F_3)\right)
\end{align*}
with real valued components defined as 
\begin{align*}
    F_n(p,\kappa,T) &\coloneqq (\mathbf{P}_p \mathbf{f}^T,\mathbf{g}_p^n)_{\mathcal{H}^k} + \left(\mathbf{P}_p \mathbf{r}\left(p,\frac{T}{T_0}\right),\mathbf{g}_p^n \right)_{\mathcal{H}^k} + \left(\mathbf{P}_{0,p}\int_{0}^{\infty}\mathbf{N}(\mathbf{q}_{p,\kappa,T}(\tau'))\,d\tau',\mathbf{g}_p^n \right)_{\mathcal{H}^k} \\
    &+ \left(\mathbf{L}_p\mathbf{P}_{0,p}\int_{0}^{\infty}(-\tau')\mathbf{N}(\mathbf{q}_{p,\kappa,T}(\tau'))\,d\tau',\mathbf{g}_p^n \right)_{\mathcal{H}^k} + \left(\mathbf{P}_{1,p}\int_{0}^{\infty}e^{-\tau'}\mathbf{N}(\mathbf{q}_{p,\kappa,T}(\tau'))\,d\tau',\mathbf{g}_p^n \right)_{\mathcal{H}^k}
\end{align*}
for $n=1,2,3$, where we note $F_n(p,\kappa,T) = \ell_{p,\kappa,T}(\mathbf{g}_p^n)$. To bound each component we use Lemma \ref{nonlinear estimates} and Lemma \ref{initial_data_lemma} to obtain 
\begin{align*}
    |F_n(p,\kappa,T)| &\lesssim \|\mathbf{g}_p^n\|_{\mathcal{H}^k}\left(\|\mathbf{f}^T\|_{\mathcal{H}^k} + \left\|\mathbf{r}\left(p,\frac{T}{T_0}\right)\right\|_{\mathcal{H}^k} + \|\mathbf{q}_{p,\kappa,T}\|_{\mathcal{X}^k}^2\right) \\
    &\lesssim \frac{\varepsilon_2}{C_2^2} + \frac{\varepsilon_2^2}{C_2^2} + \varepsilon_2^2 = \frac{\varepsilon_2}{C_2}\left(\frac{1}{MC'} + \frac{\varepsilon'}{M^2 C'} + \varepsilon' C' \right) 
\end{align*}
whenever $p\in \overline{B}_{\frac{\varepsilon_2}{C_2}}(p_0)$, $T\in \overline{B}_{\frac{T_0\varepsilon_2}{C_2}}(T_0)$ and $\kappa \in \overline{B}_{\frac{\varepsilon_2}{C_2}}(\kappa_0)$. We can now choose $M$ and $C'$ large enough and $\varepsilon'$ small enough so that $F$ is a self map. By Brouwer's fixed point theorem, we obtain a fixed point $(p^*,\kappa^*,T^*)\in \overline{B}_{\frac{\varepsilon_2}{C_2}}(p_0) \times \overline{B}_{\frac{\varepsilon_2}{C_2}}(\kappa_0) \times \overline{B}_{\frac{T_0 \varepsilon_2}{C_2}}(T_0)$ of $F$, satisfying 
\begin{align*}
    p^* &= p_0 + F_1(p^*,\kappa^*,T^*) = p^* + \ell_{p^*,\kappa^*,T^*}(\mathbf{g}_{p^*}^1) \\
    \kappa^* &= \kappa_0 + p^*\left(\frac{T^*}{T_0}-1\right) -\frac{p^*}{2(1-p^*)}(p_0-p^*) - F_2(p^*,\kappa^*,T^*) = \kappa^* - \ell_{p^*,\kappa^*,T^*}(\mathbf{g}_{p^*}^2) \\
    T^* &= T_0(1+\sqrt{1-p^*}F_3(p^*,\kappa^*,T^*)) = T^* + T_0\sqrt{1-p^*}\ell_{p^*,\kappa^*,T^*}(\mathbf{g}_{p^*}^3)
\end{align*}
which implies $\ell_{p^*,\kappa^*,T^*} \equiv 0$ on $\text{ran}({\mathbf{P}_p})$. 
\end{proof}
\subsection{Proof of Theorem \ref{Stability_main_theorem}}
\begin{proof}
    Fix $0<\delta<\omega_0$ with $\omega_0$ from Proposition \ref{spectrum_prop}, and let $\omega_1 = \omega_0 -\delta$. We take $\varepsilon_2>0$ and $C_2>1$ from Proposition \ref{Cramer_Brouer_prop}. Then for any real-valued $(f,g)\in H^{k+1}(\mathbb{R})\times H^k(\mathbb{R})$ with $\|f\|_{H^{k+1}(\mathbb{R})}+\|g\|_{H^k(\mathbb{R})} \leq \frac{\varepsilon_2}{C_2}$ there exists a unique mild solution $\mathbf{q}_{p^*,\kappa^*,T^*} = \left(q_{p^*,\kappa^*,T^*}^{(1)},q_{p^*,\kappa^*,T^*}^{(2)}\right) \in C([0,\infty);\mathcal{H}^k(-1,1))$ satisfying \eqref{No_correction_duhamel} that is a jointly classical solution to the abstract Cauchy problem \eqref{abstract_cauchy_problem}. We now reconstruct the solution to the original Cauchy problem. Recall we had 
    \begin{equation*}
        U(\tau,y) = U_{p^*,1,\kappa^*}(\tau,y) + q_{p^*,\kappa^*,T^*}^{(1)}(\tau,y).
    \end{equation*}
    Transforming variables back
    \begin{equation*}
        u(x,t) = U_{p^*,1,\kappa^*}\left(-\log\left(1-\frac{t}{T^*}\right),\frac{x-x_0}{T^*-t}\right) + q_{p^*,\kappa^*,T^*}^{(1)}\left(-\log\left(1-\frac{t}{T^*}\right),\frac{x-x_0}{T^*-t} \right)
    \end{equation*}
    is the unique solution to \eqref{PDE}, with initial data, by construction of $\mathbf{U}_{p,\kappa,T}(f,g)$, given by 
    \begin{align*}
        u(x,0) &= u_{p_0,1,\kappa,x_0,T_0}(x,0) + f(x) \\
        \partial_t u(x,0) &= \partial_t u_{p,1,\kappa,x_0,T}(x,0) + g(x)
    \end{align*}
    for $x\in B_{T^*}(x_0)$. Moreover, by Proposition \ref{Cramer_Brouer_prop}, we have 
    \begin{align*}
        \|q_{p^*,\kappa^*,T^*}^{(1)}(\tau,\cdot)\|_{H^{k+1}(-1,1)} &\leq \varepsilon_2 e^{-\omega_1 \tau} \\
        \|q_{p^*,\kappa^*,T^*}^{(2)}(\tau,\cdot)\|_{H^k(-1,1)}&\leq \varepsilon_2e^{-\omega_1 \tau}
    \end{align*}
    which implies, by the scaling of the homogeneous semi-norms 
    \begin{align*}
        (T^*-t)^{-\frac{1}{2}+s}\|u(\cdot,t)-u_{p^*,1,\kappa^*,x_0,T}(\cdot,t)\|_{\dot H^s(B_{T^*-t}(x_0))} &= \left\|q_{p^*,\kappa^*,T^*}^{(1)}\left(-\log\left(1-\frac{t}{T^*}\right),\cdot\right)\right\|_{\dot H^s(-1,1)}\\
        &\leq \varepsilon_2 (T^*-t)^{\omega_0-\delta}
    \end{align*}
    for $s=0,1,\dots,k+1$, and 
    \begin{align*}
        (T^*-t)^{-\frac{1}{2}+s}\|\partial_t u(\cdot,t)-\partial_t u_{p^*,1,\kappa^*,x_0,T}(\cdot,t)\|_{\dot H^{s-1}(B_{T^*-t}(x_0))} &= \left\|q_{p^*,\kappa^*,T^*}^{(2)}\left(-\log\left(1-\frac{t}{T^*}\right),\cdot \right)\right\|_{\dot H^{s-1}(-1,1)}\\
        &\leq \varepsilon_2 (T^*-t)^{\omega_0-\delta}
    \end{align*}
    for $s=1,\dots,k$. 
\end{proof}
\section*{Acknowledgments}
The author would like to thank Professor Manuel del Pino and Professor Monica Musso for their supervision and helpful discussions.

\appendix
\section{Frobenius Analysis}\label{sec:appendix-regularity}
The eigenequation \eqref{eigenequation} may be written in standard Heun ODE form by letting $\varphi(y) =\phi(z)$ for $y=2z-1$, 
\begin{equation}
    \phi'' + \left(\frac{\gamma}{z} + \frac{\delta}{z-1} + \frac{\varepsilon}{\left(z-\left(\frac{\sqrt{1-p}-1}{2\sqrt{1-p}}\right) \right)} \right)\phi' +\left(\frac{\lambda(\lambda+1)z-q}{z(z-1)\left(z-\left(\frac{\sqrt{1-p}-1}{2\sqrt{1-p}}\right)\right)}\right)\phi = 0
\label{Eigenequation_as_Heun}
\end{equation}
with $\gamma = \lambda -\sqrt{1-p}, \delta = \lambda +\sqrt{1-p}, \varepsilon =2$ and $q=(\sqrt{1-p}-1)\lambda^2 + (\sqrt{1-p}-1+2p)\lambda$. 
By the Frobenius theorem, we can always find a local power series solution near the singular points by substituting 
\begin{equation*}
    \phi(z) = (z-z_0)^s\sum_{k=1}^{\infty}a_k(z-z_0)^k.
\end{equation*}
Indeed, after substituting and matching the lowest order term, we obtain  indicial polynomials for $z_0\in \{0,1\}$ respectively. We have
\begin{align*}
    P_0(s) = s(s-1+\gamma) \quad \text{for} \quad z_0 =0 \\
    P_1(s) = s(s-1+\delta) \quad \text{for} \quad z_0 =1
\end{align*}
\begin{proposition}
    For $p\in (0,1)$, $\lambda \in \mathbb{C}$, with $\mathfrak{Re}\lambda \geq 0$ and $k\geq 4$, then any $H^{k+1}(0,1)$ solution to \eqref{Eigenequation_as_Heun} is in $C^\infty[0,1]$. Moreover, the set of local solutions around $z_0=1$ is one dimensional. 
\label{one_dimentsional space appendix}
\end{proposition}
\begin{proof}
    First we do the Frobenius analysis around $z_0 = 0$. We have the two roots to $P_0(s)$, $s_{\pm}\in \{0,1-\lambda +\sqrt{1-p}\}$ and the two cases:
    \begin{enumerate}
       \item $s_{+} = 1 - \lambda + \sqrt{1-p}$, $s_{-} = 0$ if $0\leq \mathfrak{Re}\lambda \leq 1+\sqrt{1-p}$
        \item $s_{+} =0$, $s_{-}=1-\lambda +\sqrt{1-p}$ if $\mathfrak{Re}\lambda > 1+\sqrt{1-p}$  
    \end{enumerate}
    
    We now have the following sub-cases.
    
    \noindent Case 1.1: $s_{+}-s_{-} \notin \mathbb{N}$ and $0\leq \mathfrak{Re}\lambda \leq 1+\sqrt{1-p}$, then there are the two independent local solutions are 
    \begin{align*}
        \phi_{11}(z) &= z^{s_{+}}h_{+}(z) \\
        \phi_{12}(z) &= h_{-}(z)
    \end{align*}
    where $h_{\pm}$ are analytic functions around $z=0$. However $\phi_{11}\notin H^{4+1}(0,1)$ (and therefore no larger regularity), since $s_{+}\notin \mathbb{N}$, and $\mathfrak{Re}\,s_{+}-5<0$ for $\mathfrak{Re}\lambda \geq 0$. Thus any local $H^{k+1}(0,1)$ solution must be a multiple of $\phi_{12}$, and hence is smooth and analytic at $0$. \\

   \noindent  Case 1.2: $s_{+}-s_{-} \in \mathbb{N}$ and $0\leq \mathfrak{Re}\lambda \leq 1+\sqrt{1-p}$, the two independent local solutions 
   \begin{align*}
       \phi_{13}(z) &= z^{s_{+}}h_{+}(z) \\
       \phi_{14}(z) &= h_{-}(z) + c\,\phi_{13}(z)\log(z)
   \end{align*}
   with $h_{\pm}$ analytic around $z=0$, and $c\in \mathbb{C}$ which may be zero unless $s_{+}=s_{-}=0$. If $c\neq 0$, then $\phi_{14} \notin H^5(0,1)$ due to the $\log(z)$ term, and so any $H^{k+1}$ solution would be a multiple of $\phi_{13}$. If $c=0$, the solution is a multiple of $\phi_{13}$ or $\phi_{14}$ and so in either case the local solution is smooth and analytic at $z=0$. \\
   
   \noindent Case 2.1: $s_{+}=0, s_{-}=1-\lambda+\sqrt{1-p}$, and $\mathfrak{Re}\lambda > 1 + \sqrt{1-p}$, the two independent local solutions are 
   \begin{align*}
       \phi_{21}(z) &= h_{+}(z) \\
       \phi_{22}(z) &= z^{s_{-}}h_{-}(z) + c\,\phi_{21}(z)\log(z).
   \end{align*}
   Note that $\phi_{22}\notin H^5(0,1)$, regardless of the value of $c$, since $s_{-}<0$. Thus any local $H^{k+1}$ solution must be a multiple of $h_{+}$ and is thus smooth and analytic at $z=0$. 

   We now do the analysis around $z_0 = 1$, for which we have $P_1(s) = s(s-1+\delta)$ and roots $s_{\pm}\in \{0,1-\lambda-\sqrt{1-p}\}$. We have the two cases 
   \begin{enumerate}
       \item $s_{+} = 1-\lambda - \sqrt{1-p}, s_{-}=0$ if $0\leq \mathfrak{Re}\lambda \leq 1-\sqrt{1-p}$ 
       \item $s_{+}=0, s_{-} = 1-\lambda - \sqrt{1-p}$ if $\mathfrak{Re}\lambda > 1-\sqrt{1-p}$
   \end{enumerate}
   We now have the following sub-cases. 
   
   \noindent Case 1.1: $s_{+}-s_{-}\notin \mathbb{N}$ and $0\leq \mathfrak{Re}\lambda\leq 1-\sqrt{1-p}$. The two independent local solutions are 
   \begin{align*}
       \widetilde{\phi}_{11}(z) &= (z-1)^{s_{+}}h_{+}(z-1) \\
       \widetilde{\phi}_{12}(z) &= h_{-}(z-1)
   \end{align*}
   with $h_{\pm}$ analytic functions around $0$. However, $\widetilde{\phi}_{11} \notin H^5(0,1)$, since $s_{+}\notin \mathbb{N}$ and $\mathfrak{Re}\,s_{+} - 5 <0$. Thus any $H^{k+1}$ solution must be a multiple of $\widetilde{\phi}_{12}$ and hence smooth and analytic at $z=1$.\\

   \noindent Case 1.2: $s_{+}-s_{-} \in \mathbb{N}$, and $0\leq \mathfrak{Re}\lambda \leq 1-\sqrt{1-p}$, then the two independent local solutions are 
   \begin{align*}
       \widetilde{\phi}_{13}(z) &= (z-1)^{s_{+}}h_{+}(z-1) \\
       \widetilde{\phi}_{14}(z) &= h_{-}(z-1) + c\, \widetilde{\phi}_{13}(z-1)\log(z-1).
   \end{align*}
   Since $s_{+}=1-\lambda-\sqrt{1-p} < 1-\lambda \leq 1$ and $s_{+}\in \mathbb{N}$, then $s_{+}=s_{-}=0$. Thus $c\neq 0$, and so $\widetilde{\phi}_{14}\notin H^5(0,1)$. Thus any local $H^{k+1}$ solution must be a multiple of $\widetilde{\phi}_{13} = h_{+}(z-1)$ which is smooth and analytic around 1. 

   \noindent Case 2.1: $s_{+}=0, s_{-} = 1-\lambda - \sqrt{1-p}$ and $\mathfrak{Re}\lambda > 1-\sqrt{1-p}$. The two independent local solutions are 
   \begin{align*}
       \widetilde{\phi}_{21}(z) &= h_{+}(z-1) \\
       \widetilde{\phi}_{22}(z) &= (z-1)^{s_{-}} + c\,\widetilde{\phi}_{21}(z-1)(z)\log(z-1).
   \end{align*}
   We note $\widetilde{\phi}_{22}\notin H^5(0,1)$ since $s_{-}<0$, regardless of the value of $c$. Thus any $H^5$ local solution is smooth and analytic at 1. \\

   The above analysis shows that any local $H^{k+1}$ solution to \eqref{Eigenequation_as_Heun} must be $C^{\infty}$, \textit{and} analytic in a neighbourhood of $z=0,1$. In addition, since \eqref{Eigenequation_as_Heun} is a second order elliptic equation with smooth coefficients on $(0,1)$, noting that the other singular point $z=\frac{\sqrt{1-p}-1}{2\sqrt{1-p}}$ does not lie in $(0,1)$, any $H^{k+1}(0,1)$ solution must be in $C^\infty(0,1)$ by elliptic regularity theory. Together with the former point, we have that the solution must be in $C^\infty[0,1]$. Moreover, the above analysis also shows that in each case the solution at each singular point is a multiple of the same analytic function, and thus the set is one-dimensional.  
\end{proof}

\section{Non-existence of other generalised eigenfunctions}
\begin{lemma}
    There do not exist $\mathbf{v}\in \mathcal{H}^4$ such that 
    $\mathbf{L}_{\frac{3}{4}}\mathbf{v} = \mathbf{g}_{0,\frac{3}{4}}$.
    \label{no_more_generalised_eigfunctions}
\end{lemma}

\begin{proof}
    Assume for contradiction that there exists $(v_1,v_2)\in H^4(-1,1)\times H^3(-1,1)$ satisfying
    \begin{align*}
        \begin{cases}
            -y\partial_y v_1 + v_2 = g_{0,\frac{3}{4},1} \equiv -\log\!\left(1+\frac{y}{2}\right) - \frac{3}{2+y},\\[0.5em]
            \partial_{yy}v_1 - y\partial_y v_2 + \Big(\frac{3}{2+y}-1\Big)v_2 = g_{0,\frac{3}{4},2} \equiv \frac{5y+4}{(2+y)^2}.
        \end{cases}
    \end{align*}
    Substituting $v_2 = y\partial_y v_1 + g_{0,\frac{3}{4},1}$ into the second equation gives
    \begin{equation*}
        (1-y^2)\partial_{yy}v_1 -\frac{y(1+2y)}{y+2}\partial_y v_1 = g(y)
        \coloneqq \frac{1-y}{2+y}\log\!\left(1+\frac{y}{2}\right) + \frac{-y^2+3y+7}{(2+y)^2}.
    \end{equation*}
    Introducing the integrating factor
    \[
        p(y) = (y+2)^2\!\left(\frac{1-y}{1+y}\right)^{\!1/2},
    \]
    the equation becomes
    \[
        \partial_y v_1(y) = c\,p(y)^{-1} 
        + p(y)^{-1}\int_{y}^{1}\frac{p(z)g(z)}{1-z^2}\,dz.
    \]
    Computing the integral explicitly yields
    \begin{equation*}
        \partial_y v_1(y) =
        \frac{\sqrt{y+1}\Big(c+3\sqrt{3}\,\arctan\!\Big(\tfrac{2y+1}{\sqrt{3-3y^{2}}}\Big)\Big)}{\sqrt{1-y}\,(y+2)^{2}}
        + \frac{-y\log(2)+(y-1)\log(y+2)-3+\log(2)}{(y+2)^{2}}.
    \end{equation*}

    \vspace{0.5em}
    \noindent\textbf{Asymptotics near $y=1$.}  
    Let $t=1-y\downarrow 0$.  
    Using
    \[
        \sqrt{\tfrac{1-y}{1+y}} \sim \frac{\sqrt{t}}{\sqrt{2}}, 
        \qquad
        \frac{2y+1}{\sqrt{3-3y^2}} \sim \frac{3}{\sqrt{6t}},
    \]
    we have
    \[
        \arctan\!\Big(\tfrac{2y+1}{\sqrt{3-3y^2}}\Big)
        = \frac{\pi}{2} - \sqrt{\tfrac{2}{3}}\sqrt{t} + O(t^{3/2}).
    \]
    Hence,
    \begin{align*}
        3\sqrt{3}\,\arctan\!\Big(\tfrac{2y+1}{\sqrt{3-3y^2}}\Big)
        &= \tfrac{3\sqrt{3}\pi}{2} - 3\sqrt{2}\sqrt{t} + O(t^{3/2}),\\[0.3em]
        p(y) &\sim \tfrac{9}{\sqrt{2}}\sqrt{t}.
    \end{align*}
    Therefore,
    \[
        \partial_y v_1(y) \sim 
        \frac{\sqrt{2}}{9\sqrt{t}}\left(c+\tfrac{3\sqrt{3}\pi}{2}\right)
        + O(1).
    \]
    For $\partial_y v_1\in L^2(-1,1)$, the singular term $\sim (1-y)^{-1/2}$ must vanish, forcing
    \[
        c = -\frac{3\sqrt{3}}{2}\pi.
    \]

    \vspace{0.5em} 
    With this choice of $c$, the leading singular term cancels and
    \[
        \partial_y v_1(y) = A + B\sqrt{1-y} + O(1-y),
        \qquad y\to 1^-,
    \]
    for some $A,B\neq 0$. Differentiating gives
    \[
        \partial_{yy}v_1(y) \sim -\frac{B}{2}(1-y)^{-1/2},
    \]
    so $|\partial_{yy}v_1(y)|^2\sim C(1-y)^{-1}$, which is not integrable near $y=1$.  
    Hence $\partial_{yy}v_1\notin L^2(-1,1)$, contradicting $v_1\in H^2(-1,1)$.  
    Therefore, such $\mathbf{v}$ cannot exist.
\end{proof}
\begin{lemma}
    There do not exist $\mathbf{u}\in \mathcal{H}^4$ such that 
    $(\mathbf{L}_{\frac{3}{4}}-\mathbf{I})\mathbf{u} = \mathbf{f}_{1,\frac{3}{4}}$.
    \label{rank_1 of 0 eigenvalue}
\end{lemma}

\begin{proof}
    Writing out the system and eliminating $u_2$ as before yields
    \begin{equation*}
        (1-y^2)\partial_{yy}u_1 -\frac{y(4y+5)}{y+2}\partial_y u_1 
        - \frac{1+2y}{y+2}u_1 
        = \frac{-3(y+3)}{4(y+2)^2}.
    \end{equation*}
    A direct computation gives the general solution
    \begin{align*}
        u_1(y)
        &= \frac{3\log(y+2)}{4(y+2)}
        + \frac{c_1}{y+2}
        + \frac{c_2\sqrt{1+y}}{\sqrt{1-y}(y+2)}\\[0.3em]
        &\quad + \frac{3\sqrt{3}\sqrt{y+1}}{4\sqrt{1-y}(y+2)}
        \arctan\!\Big(\frac{\sqrt{3}\sqrt{1-y^2}}{2y+1}\Big).
    \end{align*}

    Near $y=1$, the term with $c_2$ behaves like $(1-y)^{-1/2}$ and thus 
    forces $c_2=0$ for $u_1\in L^2(-1,1)$.  
    We now inspect the remaining $\arctan$ term.

    The argument $\frac{\sqrt{3}\sqrt{1-y^2}}{2y+1}$ diverges as 
    $y\to -\frac12$, with
    \[
        \lim_{y\to(-1/2)^-}\frac{\sqrt{3}\sqrt{1-y^2}}{2y+1} = -\infty,
        \qquad
        \lim_{y\to(-1/2)^+}\frac{\sqrt{3}\sqrt{1-y^2}}{2y+1} = +\infty.
    \]
    Hence $\arctan(\cdot)$ has a jump discontinuity of size $\pi$ across $y=-\tfrac12$, and since the pre-factor is finite and nonzero at $y=-\tfrac12$, $u_1$ itself has a finite jump there.
    Therefore $u_1$ is discontinuous and thus not in $H^1(-1,1)$, 
    contradicting the assumption that $\mathbf{u}\in \mathcal{H}^4$.  
    The claim follows.
\end{proof}
\bibliographystyle{plain}
\bibliography{paper_draft}
\end{document}